\documentclass[12pt]{article}

\usepackage{amsmath}
\usepackage{amsthm}
\usepackage{amssymb}
\usepackage{amsfonts}
\usepackage{array}
\usepackage{hyperref}
\usepackage{comment}

\newtheorem{theorem}{THEOREM}
\newtheorem{corollary}{COROLLARY}
\newtheorem{proposition}{PROPOSITION}
\newtheorem{meta-proposition}{META-PROPOSITION}
\newtheorem{lemma}{LEMMA}
\newtheorem{definition}{DEFINITION}
\newtheorem{meta-definition}{META-DEFINITION}

\newtheorem{example}{EXAMPLE}

\newcommand{\ol}[1]{\overline{#1}}

\usepackage{stackengine}
\stackMath
\usepackage{graphicx}
\def\vecsign{\mathchar"017E}
\def\dvecsign{\smash{\stackon[-2.35pt]{\vecsign}{\rotatebox{180}{$\vecsign$}}}}
\def\rsym#1{\def\useanchorwidth{T}\stackon[-5.2pt]{\mathbf{#1}}{\dvecsign}}
\def\rasy#1{\def\useanchorwidth{T}\stackon[-5.2pt]{\mathbf{#1}}{\!\vecsign}}

\begin{document}

\title{Consistency between transitive relations\\ and between cones}

\author{Tom Fischer\thanks{Institute of Mathematics, University of Wuerzburg, 
Emil-Fischer-Strasse 30, 97074 Wuerzburg, Germany.
Tel.: +49 931 3188911.
E-mail: {\tt tom.fischer@uni-wuerzburg.de}.
}\\
University of Wuerzburg}
\date{\today}

\maketitle

\begin{abstract}
A relation extends another relation consistently 
if its symmetric, respectively its asymmetric, part contains
the corresponding part of the smaller relation.
It is shown that there exists no finite circular chain made from two transitive relations $\mathbf{A}$ and $\mathbf{B}$
with at least one link from their asymmetric parts
if and only if there exists a total preorder which consistently extends both. 
Additionally, this extension is uniquely determined if and only if 
the reflexive transitive closure of the union of $\mathbf{A}$ and $\mathbf{B}$ is total.
Applications:
(1)
If the steps of a walk come from two positive cones, with at least one step from one of the cones' non-linear parts, 
then returning to the origin is impossible if and only if there exists a third cone of which the linear part contains each 
of the linear parts of the two original cones, and of which the non-linear part contains each of the two non-linear parts.
(2)
Reminiscent of the Fundamental Theorem of Asset Pricing,
absence of arbitrage is equivalent to the existence of 
a complete preference order which consistently extends a market's fair exchange relation and its 
objective strict preference order.
(3)
Another impossibility in microeconomics: For two agents, or an agent and a market, 
with additive positively homogeneous preferences, 
an allocation can never be Pareto optimal if they are not `cut from the same cloth'.
\end{abstract}





\section{Intentions}
\label{sec:1}

Any relation splits into its symmetric and its asymmetric part.
For instance, a preorder (quasiorder) is the disjoint union of an equivalence relation and a strict partial order.
While this is trivial, it is easy to find examples of disjoint equivalence relations and strict partial orders
whose union is neither a preorder, nor even just a transitive relation.
The question of what could manifest a meaningful reverse direction of the splitting procedure for --
more generally than preorders -- transitive relations therefore seems to deserve closer examination.

In my literature research, I could not find any related previous results. Apart from partial orders,
transitive relations or preorders generally do not seem to draw a lot of attention in mathematics. 
At least in economic theory, preorders do, and are usually referred to
as preference orders. However, the focus there typically lies on questions 
that I will not consider here (e.g.~\cite{MaAn}). Textbooks -- as extensive and thorough they may be on orders and lattices 
(e.g.~\cite{RoSt} or \cite{ScBe}) -- seem to mostly skip over the topic of transitive relations and preorders, as well.

Diving further into the subject, it became clear to me that to meaningfully construct a preorder from a given equivalence relation
and a strict partial order first of all means that the given equivalence relation needs to be contained in (i.e.~imply) the one 
derived from (or implied by) the constructed preorder, and, similarly, the given strict partial order needs to be contained in (imply) the derived one.
Another observation was that, for the given equivalence relation and the given strict partial order, no finite closed, i.e.~circular, chain 
containing both relations should exist, as this in a quite obvious sense means that they are not consistent with one another. 
A corollary of my main result below states that these two conditions are, indeed, equivalent.

More generally, my main theorem states that two transitive relations are chain-consistent with one another 
in the sense that there exists no finite circular chain made from them
with at least one link from their asymmetric parts
if and only if there exists a total preorder which consistently extends both of them 
in the sense that the symmetric, respectively the asymmetric part of each of the two given
relations lies in the corresponding part of the extending preorder.
Expressed differently, two transitive relations are consistent if and only if they both can be meaningfully
embedded into a third one which also possesses a completeness property.
Additionally, the extending third relation is uniquely determined if and only if 
the reflexive transitive closure of the union of the two original relations is total.
These results generally contribute to the understanding of transitive relations, but also more
specifically to the understanding of the relationship between preorders, equivalence relations, 
and strict partial orders.

I then provide three applications of my results. One in geometry or topological vector spaces,
one in financial economics, and the third in microeconomics.

The first one is related to the fact that a walk with steps taken from a positive cone cannot end at the starting point 
if at least one step is from the cone's non-linear part, which is the cone less the largest linear space contained in it.
What if the steps come from two cones? 
My main theorem implies that if the steps of a walk come from two positive cones, with at least one step from one of the 
cones' non-linear parts, then returning to the origin is impossible if and only if there exists a third cone of which the linear 
part contains each of the linear parts of the two original cones, and of which the non-linear part contains each of the 
two non-linear parts of the original cones.
I also formulate this as follows: Two positive cones are path-consistent if and only if there exists a common consistent extension
into a third cone.
In an example, I then illustrate how a related notion of (path-)consistency between a cone and a linear functional
implies a strict type of positivity of that functional with respect to the cone.

The second example concerns financial economics or mathematical finance.
I show that if a market is modeled by an equivalence relation for its objectively fair exchanges,
meaning an indifference relation,
and by a strict partial order for its objective down-trades, meaning an objective strict
preference order, then such a market is free of arbitrage opportunities, or so-called free lunches,
if and only if there exists a complete preference order which consistently extends the 
fair exchange and the down-trade relation. I make the argument that this result is so reminiscent of the
seminal Fundamental Theorem of Asset Pricing, which also is an extension theorem, 
that it can be considered a broad generalization of it.

The third application is in microeconomics: For two agents, or an agent and a market, 
with additive positively homogeneous preferences, meaning vector preorders,
an allocation can never be Pareto optimal if the agents' preferences are not consistent
with one another in the earlier explained sense. Expressed differently, if both agents'
preferences cannot meaningfully, i.e.~consistently, be embedded into the same bigger preference order, 
they will never reach an equilibrium.

Some of the lemmas and propositions presented below are very straightforward. 
However, while being simple, 
I could not find most of the needed preliminary results on transitive relations or preorders 
in the required generality in the literature. 
For this work to be self-contained, I therefore wanted to include them. 
Where possible, I refer to the literature even if the reference only points to a related
result for partial orders.


\section{Basic properties of relations}

For the sake of notation and for the convenience of the reader, in this section, I recapitulate several basic definitions 
and properties of relations, as well as several basic results on them. In particular, the generation of relations with
certain properties is an important and later recurring topic.
In the following, all relations on a non-empty set $\mathcal{M}$ are understood to be homogeneous binary relations, meaning they are 
non-empty subsets of $\mathcal{M}\times \mathcal{M}$.

\begin{definition}
\label{def:relations}
A relation $\mathbf{R}$ on a set $\mathcal{M}$ can have one or several of the following properties.
For any $\alpha,\beta,\gamma,\delta\in\mathcal{M}$ and for $\alpha \mathbf{R} \beta$ denoting $(\alpha, \beta) \in \mathbf{R}$:
\begin{enumerate}
\item 
\label{T}
Transitivity:
If $\alpha \mathbf{R} \beta$ and $\beta \mathbf{R} \gamma$,
then $\alpha \mathbf{R} \gamma$.
\item 
\label{sym}
Symmetry: 
$\alpha \mathbf{R} \beta$ implies $\beta \mathbf{R} \alpha$.
\item 
\label{asy}
Asymmetry: 
$\alpha \mathbf{R} \beta$ implies $\beta \mathbf{R\!\!\!\!/}\, \alpha$.
\item 
\label{ref}
Reflexivity: 
$\alpha \mathbf{R} \alpha$.
\item 
\label{irr}
Irreflexivity: 
$\alpha \mathbf{R\!\!\!\!/}\; \alpha$.
\item 
\label{anti}
Antisymmetry: 
$\alpha \mathbf{R} \beta$ and $\beta \mathbf{R} \alpha$ implies $\alpha = \beta$.
\item 
\label{add}
Additivity:
For $\mathcal{M}$ a real linear space, 
$\alpha \mathbf{R} \beta$ and $\gamma \mathbf{R} \delta$ implies $(\alpha+\gamma) \mathbf{R} (\beta+\delta)$.
\item 
\label{PH}
Positive homogeneity:
For $\mathcal{M}$ a real linear space, 
$\alpha \mathbf{R} \beta$ and $r\in\mathbb{R}_0^+$ implies $(r\alpha) \mathbf{R} (r\beta)$.
\item 
\label{To}
Totality:
Either $\alpha \mathbf{R} \beta$, $\beta \mathbf{R} \alpha$, or both.
\end{enumerate}
\end{definition}

Obviously, properties of Def.~\ref{def:relations} can be mutually exclusive, or a combination of properties can 
imply another one. For instance, reflexivity with additivity implies transitivity.
At least in the specific case of partial orders (properties \ref{T}, \ref{ref}, and \ref{anti}; e.g.~\cite{ScBe}), 
the proposition below is textbook knowledge. 

\begin{proposition}
\label{prop:Intersections of relations}
Let $\mathcal{R}$ be a family of relations on $\mathcal{M}$ which all share the same non-empty 
subset of the first eight properties of Def.~\ref{def:relations}.
Assume that $\mathcal{M}$ is a real linear space if properties \ref{add} or \ref{PH} are considered.
If $\bigcap \mathcal{R}$ is non-empty, then it is a relation on $\mathcal{M}$ with the same set of properties
that the individual relations have.
\end{proposition}

\begin{proof}
For any single one of the properties \ref{T} to \ref{PH}, 
if all relations have the corresponding property, then it is straightforward to check that a non-empty intersection must have it.
\end{proof}

\begin{corollary}
\label{cor:generatedrelation}
Let $\mathbf{S}$ be an arbitrary non-empty subset of $\mathcal{M}\times \mathcal{M}$ and choose a non-empty subset of the properties 
\ref{T}, \ref{sym}, \ref{ref}, \ref{add}, and \ref{PH} of Def.~\ref{def:relations}.
Assume that $\mathcal{M}$ is a real linear space in case properties \ref{add} or \ref{PH} are chosen.
Then there exists a uniquely determined smallest relation on $\mathcal{M}$ that contains
$\mathbf{S}$ and that has the chosen properties.
\end{corollary}

Smallest here means an inclusion-minimal relation.

\begin{proof}
Since the trivial relation, $\mathcal{M}\times \mathcal{M}$, has the properties \ref{T}, \ref{sym}, \ref{ref}, \ref{add}, and \ref{PH},
the intersection of all relations on $\mathcal{M}$ which contain $\mathbf{S}$ and have the chosen properties is non-empty.
By Prop.~\ref{prop:Intersections of relations}, this intersection again has the chosen properties. 
By construction, there is no smaller such relation.
\end{proof}

\begin{definition}
\label{def:Generated relation}
The uniquely determined smallest relation of Corollary \ref{cor:generatedrelation} is called the relation with the chosen properties
on $\mathcal{M}$ generated by $\mathbf{S}$.\end{definition}

From now on, transitive relations will be denoted by the symbol $\preccurlyeq$, which can come with various indices,
or other attached symbols, to create distinctive notation.

\begin{definition}
\label{def:Transitive relations}
A transitive relation
$\preccurlyeq$ is a relation with property \ref{T} of Def.~\ref{def:relations}.
For an arbitrary non-empty $\mathbf{S} \subset \mathcal{M}\times \mathcal{M}$, 
let $\preccurlyeq_\mathbf{S}$ denote the transitive relation generated by $\mathbf{S}$.
If $\mathcal{M}$ is a real linear space, let 
$\preccurlyeq^{add,ph}_\mathbf{S}$ denote the transitive additive positively homogeneous relation generated by $\mathbf{S}$.
\end{definition}

There is a reason why I positioned transitivity at the top in Def.~\ref{def:relations}. If I think of any relation that is supposed to have some ordering
capability, transitivity seems to be the most basic property I would expect: If some $\beta$ is at least as great or as good as some $\alpha$, and some
$\gamma$ is at least as good or as great as $\beta$, then I should reasonably be able to expect $\gamma$ to be at least as great or good as $\alpha$. 
With regard to (strictly) ordering anything, there seems to be little value to a relation that is not at least transitive. This attitude is, of course, reflected
by the presence of transitivity in the contemporary definitions of preorders, partial orders, and strict orders 
(see also Def.~\ref{def:Preorder} and Def.~\ref{def:strict partial order} below).

\begin{definition}
\label{def:transitive closure}
Let $\mathbf{S}$ be an arbitrary non-empty subset of $\mathcal{M}\times \mathcal{M}$. 
The relation $\mathbf{S}^{t}$ defined by
$\alpha \mathbf{S}^{t} \beta$ for $\alpha, \beta \in \mathcal{M}$, if, for some $n\in\mathbb{N}\setminus\{0,1\}$,
$(\gamma_i,\gamma_{i+1})\in \mathbf{S}$ for $i=1,\ldots,n-1$ and $\alpha = \gamma_1$ and $\beta = \gamma_n$,
is called the transitive closure of $\mathbf{S}$.
\end{definition}


I omit the proof for the following textbook result (cf.~\cite{RoSt}, \cite{ScBe}).

\begin{proposition}
\label{prop:transitive closure}
The transitive relation generated by $\mathbf{S}$ is its transitive closure, 
\begin{eqnarray}
\preccurlyeq_\mathbf{S}
& = &
\mathbf{S}^{t} .
\end{eqnarray}
\end{proposition}



\begin{definition}
\label{def:Preorder}
A preorder is a reflexive transitive relation (\ref{T} and \ref{ref} of Def.~\ref{def:relations}).
\end{definition}

Historically, for instance in the important publication \cite{MaHo} by McNeille, preorders were also simply called `orders'.
In the contemporary literature (cf.~\cite{RoSt}), partial orders are defined as antisymmetric preorders
(\ref{T}, \ref{ref}, and \ref{anti} of Def.~\ref{def:relations}). The use of the term quasiorder for a preorder seems to
be less common (cf.~\cite{RoSt}, \cite{ScBe}).

\begin{definition}
\label{def:VectorPreorder}
A vector preorder is an additive, positively homogeneous preorder (\ref{T}, \ref{ref}, \ref{add}, and \ref{PH} of Def.~\ref{def:relations}).
\end{definition}

\begin{lemma}
\label{lem:1lkm}
A total transitive relation (\ref{T} and \ref{To} of Def.~\ref{def:relations}) is reflexive and, thus, a total preorder.
\end{lemma}

\begin{definition}[Reflexive closure]
\label{def:Reflexive closure}
Let $\mathbf{S}$ be an arbitrary non-empty subset of $\mathcal{M}\times \mathcal{M}$. 
The relation $\mathbf{S}^{r} = \mathbf{S} \cup \{(\alpha, \alpha) : \alpha \in \mathcal{M}\}$
is called the reflexive closure of $\mathbf{S}$.
\end{definition}

By definition, $\mathbf{S}^{r}$ is the smallest reflexive relation containing $\mathbf{S}$. 
It is straightforward to see that the generation of the transitive closure and of the reflexive closure is commutative,
i.e.~$(\mathbf{S}^{r})^t = (\mathbf{S}^{t})^r$.
Thus, $\preccurlyeq^r_\mathbf{S}$ is the uniquely determined smallest preorder which contains $\mathbf{S}$.

\begin{lemma}
\label{lem:3mjy}
Let $\preccurlyeq_1$ and $\preccurlyeq_2$ be two vector preorder. 
Then, $\preccurlyeq^{add,ph}_{\preccurlyeq_1 \cup \, \preccurlyeq_2}$ is a vector preorder,
and $\preccurlyeq^{add,ph}_{\preccurlyeq_1 \cup \, \preccurlyeq_2} \; = \; \preccurlyeq_{\preccurlyeq_1 \cup \, \preccurlyeq_2}$.
\end{lemma}

Therefore, with Prop.~\ref{prop:transitive closure}, the transitive additive positively homogeneous relation generated by two
vector preorders is the transitive closure of their union.

\begin{proof}[Proof of Lemma \ref{lem:3mjy}]
The relation $\preccurlyeq_{\preccurlyeq_1 \cup \, \preccurlyeq_2}$ is a preorder, because it is transitive and, with 
$\preccurlyeq_1$ and $\preccurlyeq_2$, it again is reflexive. It now suffices to show that 
$\preccurlyeq_{\preccurlyeq_1 \cup \, \preccurlyeq_2}$ is additive and positively homogeneous. Positive homogeneity follows 
via Prop.~\ref{prop:transitive closure}
from Def.~\ref{def:transitive closure}, because the chain in that definition can be scaled by any $r\in\mathbb{R}_0^+$.
Assume now $(\alpha, \beta), (\ol\alpha, \ol\beta)\in \,\preccurlyeq_{\preccurlyeq_1 \cup \, \preccurlyeq_2}$.
By Def.~\ref{def:transitive closure} and Prop.~\ref{prop:transitive closure},
$(\alpha, \beta)\in \,\preccurlyeq_{\preccurlyeq_1 \cup \, \preccurlyeq_2}$ means that 
for some $n\in\mathbb{N}\setminus\{0,1\}$,
$(\gamma_i,\gamma_{i+1})\in (\preccurlyeq_1 \cup \preccurlyeq_2)$ for $i=1,\ldots,n-1$ and $\alpha = \gamma_1$ and $\beta = \gamma_n$.
I can add $(\ol\alpha,\ol\alpha)$ to each link of this chain, because of reflexivity and additivity of both, $\preccurlyeq_1$ and $\preccurlyeq_2$.
The so obtained chain with links in $\preccurlyeq_1 \cup \preccurlyeq_2$
therefore starts with $(\alpha+\ol\alpha, \gamma_2+\ol\alpha)$ and ends with $(\gamma_{n-1}+\ol\alpha, \beta+\ol\alpha)$.
Similarly, there exists a chain for $(\ol\alpha, \ol\beta)\in \,\preccurlyeq_{\preccurlyeq_1 \cup \, \preccurlyeq_2}$ that
starts with $(\beta+\ol\alpha, \beta + \ol\gamma_2)$ and ends with $(\beta+\ol\gamma_{\ol n-1}, \beta+\ol\beta)$ for 
appropriate $\ol n$ and $(\ol\gamma_i,\ol\gamma_{i+1})\in (\preccurlyeq_1 \cup \preccurlyeq_2)$ for $i=1,\ldots,\ol n-1$.
Thus, summarily, a chain with links in $\preccurlyeq_1 \cup \preccurlyeq_2$ exists that starts with $(\alpha+\ol\alpha, \gamma_2+\ol\alpha)$
and ends with $(\beta+\ol\gamma_{\ol n-1}, \beta+\ol\beta)$. Therefore, 
$(\alpha+\ol\alpha, \beta+\ol\beta)\in \,\preccurlyeq_{\preccurlyeq_1 \cup \, \preccurlyeq_2}$,
and additivity is established.
\end{proof}


\section{Equivalence relations and strict partial orders}
\label{sec:equivalence relations and strict partial orders}

In this section, and next to providing more definitions, I derive several results
on calculus with relations, as well as some statements concerning the splitting and the merging 
of transitive relations.

\begin{definition}
\label{def:Transitive relations2}
For any relation $\mathbf{R}$,
\begin{eqnarray}
\label{eq:2aswA}
\rsym{R}
& =  & 
\{(\alpha,\beta) \in \mathbf{R} : \alpha \mathbf{R} \beta \text{ and } \beta \mathbf{R} \alpha\} 
\end{eqnarray}
is its symmetric part, and 
\begin{eqnarray}
\label{eq:2asw}
\rasy{R}
& =  & 
\mathbf{R} \setminus \rsym{R}
\end{eqnarray}
its asymmetric part.
For a transitive relation $\preccurlyeq$, I denote $\sim \; = \rsym{\preccurlyeq}$ and $\prec \; = \rasy{\preccurlyeq}$.
\end{definition}

Obviously, one of the parts \eqref{eq:2aswA} and \eqref{eq:2asw} can be empty. 

Formulated for total transitive relations, Prop.~1.B.1 in \cite{MaAn} implies the first statement of the following lemma.

\begin{lemma}
\label{lem:1csv}
Let $\preccurlyeq$ be a transitive relation.
\begin{enumerate}
\item
If $\alpha \sim \beta$ and $\beta \sim \gamma$, then $\alpha \sim \gamma$.
If (a) $\alpha \prec \beta$ and $\beta \sim \gamma$, or
if (b) $\alpha \sim \beta$ and $\beta \prec \gamma$, or
if (c) $\alpha \prec \beta$ and $\beta \prec \gamma$, then $\alpha \prec \gamma$.
\item
Let $\preccurlyeq$ also be reflexive and additive.
If $\alpha \sim \beta$ and $\gamma \sim \delta$, then $\alpha + \gamma \sim \beta + \delta$.
If $\alpha \prec \beta$ and $\gamma \sim \delta$, or
if $\alpha \prec \beta$ and $\gamma \prec \delta$, then $\alpha + \gamma \prec \beta + \delta$.
\item
Let $\preccurlyeq$ also be positively homogeneous, and assume $r\in \mathbb{R}_0^+$.
If $\alpha \sim \beta$, then $r\alpha \sim r\beta$.
If $\alpha \prec \beta$, then $r\alpha \prec r\beta$.
\end{enumerate}
\end{lemma}

\begin{proof}
1.
Transitivity implies for all these cases $\alpha \preccurlyeq \gamma$. 
The first statement follows, because symmetry of $\sim$ implies that the reverse holds, too.
For (a), assume $\gamma \preccurlyeq \alpha$. Now, $\alpha \preccurlyeq \beta \preccurlyeq \gamma \preccurlyeq \alpha$, 
which implies the contradiction $\alpha \sim \beta$. Statements (b) and (c) follow in a similar manner.

2. Additivity implies for all three cases $\alpha + \gamma \preccurlyeq \beta + \delta$. The first statement holds, because symmetry of $\sim$ implies that
$\beta + \delta \preccurlyeq \alpha + \gamma$, as well.
The second statement follows, because if $\beta + \delta \preccurlyeq \alpha + \gamma$ held true, then $\gamma \preccurlyeq \delta$,
$-\gamma \preccurlyeq -\gamma$, and $-\delta \preccurlyeq -\delta$ could be added to obtain the contradiction $\beta \preccurlyeq \alpha$.
The third statement follows in a similar manner.

3. Since $\alpha \preccurlyeq \beta$ and $\beta \preccurlyeq \alpha$, the first statement follows immediately from positive homogeneity.
For the second one, if $r\beta \preccurlyeq r\alpha$, then multiplication with $1/r$ would imply the contradiction $\beta \preccurlyeq \alpha$.
\end{proof}

\begin{definition}
\label{def:equivalence relation}
An equivalence relation 
is a symmetric preorder  (\ref{T}, \ref{sym}, and \ref{ref} of Def.~\ref{def:relations}).
\end{definition}

For the next definition note that asymmetry implies irreflexivity.

\begin{definition}
\label{def:strict partial order}
A strict partial order
is an asymmetric transitive relation (\ref{T}, \ref{asy}, and \ref{irr} of Def.~\ref{def:relations}).
\end{definition}

A transitive relation on $\mathcal{M}$ splits into an equivalence relation on a subset of $\mathcal{M}$ and 
a strict partial order on $\mathcal{M}$. For total transitive relations, this is (i) and (ii) of Prop.~1.B.1 in \cite{MaAn}.

\begin{lemma}
\label{lem:Implied equivalence and strict preference}
For any transitive relation $\preccurlyeq$ on a set $\mathcal{M}$, if non-empty, $\sim$ 
is a symmetric transitive relation on $\mathcal{M}$ and therefore
an equivalence relation on 
\begin{eqnarray}
\mathcal{M}_{\sim} & := & \{\alpha \in\mathcal{M} : \alpha \sim \beta \text{ for some } \beta\in\mathcal{M}\} .
\end{eqnarray}
If non-empty, 
$\prec$ is a strict partial order on $\mathcal{M}$.
If $\preccurlyeq$ has the property \ref{anti}, or simultaneously the properties \ref{ref} and \ref{add}, or the property \ref{PH} of Def.~\ref{def:relations}, 
then so do $\sim$ and $\prec$.
\end{lemma}

By \eqref{eq:2asw}, $\sim$  and $\prec$ are disjoint (mutually exclusive), i.e.~$\sim \cap \prec \; = \; \emptyset$.

\begin{proof}
$\sim$ is symmetric by definition and inherits transitivity (Lemma \ref{lem:1csv}). It therefore is reflexive, symmetric and transitive on $\mathcal{M}_{\sim}$. 
Since $\sim$ is the entire symmetric part of $\preccurlyeq$, the rest, $\prec$, if non-empty,
is asymmetric. 
Transitivity is given by Lemma \ref{lem:1csv}.
Optional antisymmetry: While $\prec$ has the property in the sense that it applies to the empty set, $\sim$ directly inherits it from $\preccurlyeq$.
Optional reflexivity and additivity, or optional positive homogeneity of $\sim$ and $\prec$: Lemma \ref{lem:1csv}.
\end{proof}

\begin{corollary}
\label{cor:2oiv}
A preorder on $\mathcal{M}$  splits into an equivalence relation on $\mathcal{M}$  and a strict partial order on $\mathcal{M}$ . 
\end{corollary}

\begin{definition}
If non-empty, $\sim$ is called the implied equivalence relation of the transitive relation $\preccurlyeq$,
and $\prec$, if non-empty, is called its implied strict partial order.
\end{definition}

The following notation can be handy:
\begin{eqnarray}
\label{eq:3lpc}
\succ
& := &
\{(\beta, \alpha) : (\alpha, \beta) \in \; \prec\} .
\end{eqnarray}

\begin{lemma}
\label{lem:eqrelpre}
An equivalence relation on a subset of $\mathcal{M}$ is a preorder on $\mathcal{M}$ with an empty implied strict partial order.
\end{lemma}

\begin{lemma}
\label{lem:cor3pom}
For any transitive relation $\preccurlyeq$, the implied strict partial order $\prec$ is a strict partial order on the
equivalence classes  of the implied equivalence relation, i.e.~on $\preccurlyeq / \sim$, in the sense that for
any $\alpha \sim \beta$ and $\gamma \sim \delta$,  $\alpha \prec \gamma$ implies $\beta \prec \delta$.
\end{lemma}

Essentially, this is Theorem 2.10 in \cite{MaHo}. Note that MacNeille in his Definition 2.4 defines $=$ in the
same manner as I defined $\sim$, but in the case of a preorder. 
From a contemporary perspective, he thus defined equivalence, not equality (which has no definition).

\begin{proof}
Lemma \ref{lem:1csv}.
\end{proof}

\begin{proposition}
\label{prop:1kop}
Let $\preccurlyeq_1$ and $\preccurlyeq_2$ be two transitive relations. 
If $\prec_2 \; \subset \; \prec_1$ and $\sim_1 \; \subset \; \sim_2$, 
then $\preccurlyeq \;  := \; \sim_1 \cup \prec_2$ is a transitive relation
with $\sim \; = \; \sim_1$ and $\prec \; = \; \prec_2$.
If $\preccurlyeq_1$ and $\preccurlyeq_2$ both
have the property \ref{ref}, or \ref{anti}, or simultaneously the properties \ref{ref} and \ref{add}, or the property \ref{PH} of Def.~\ref{def:relations}, 
then so does $\preccurlyeq$.
\end{proposition}

\begin{proof}
For transitivity, distinguish four cases.
Case 1: $\alpha \sim_1 \beta$ and $\beta \sim_1 \gamma$. This implies $\alpha \sim_1 \gamma$.
Case 2: $\alpha \sim_1 \beta$ and $\beta \prec_2 \gamma$. Since also $\alpha \sim_2 \beta$, Lemma \ref{lem:1csv}
implies $\alpha \prec_2 \gamma$.
Case 3: $\alpha \prec_2 \beta$ and $\beta \sim_1 \gamma$.  Since also $\beta \sim_2 \gamma$, Lemma \ref{lem:1csv}
implies $\alpha \prec_2 \gamma$.
Case 4: $\alpha \prec_2 \beta$ and $\beta \prec_2 \gamma$. Lemma \ref{lem:1csv} implies $\alpha \prec_2 \gamma$.
Thus, $\alpha \preccurlyeq \beta$ and $\beta \preccurlyeq \gamma$ implies $\alpha \preccurlyeq \gamma$.
Further, $\sim \; = \; \sim_1$ and $\prec \; = \; \prec_2$ must holds since $\sim_1 \cap \prec_2 \; = \emptyset$.
Optional reflexivity follows for $\preccurlyeq$ since $\sim_1$ is reflexive if $\preccurlyeq_1$ and $\preccurlyeq_2$ are.
Optional antisymmetry of $\preccurlyeq$ is inherited from $\preccurlyeq_1$.
If $\preccurlyeq_1$ and $\preccurlyeq_2$ are additive and reflexive, then a case distinction
shows that Lemma \ref{lem:1csv} and $\sim_1 \; \subset \; \sim_2$ imply these properties for $\preccurlyeq$. 
Lemma \ref{lem:1csv} also takes care of positive homogeneity.
\end{proof}

\begin{proposition}
\label{prop:3oiu}
Let $\preccurlyeq_1 \; \subsetneq \; \preccurlyeq_2$ be two total preorders. 
Then $\prec_2 \; \subsetneq \; \prec_1$ and $\sim_1 \; \subsetneq \; \sim_2$.
\end{proposition}

\begin{proof}
If $\alpha \preccurlyeq_2 \beta$ and $\beta \not\preccurlyeq_2 \alpha$, then 
$\alpha \preccurlyeq_1 \beta$ and $\beta \not\preccurlyeq_1 \alpha$ follows from the inclusion.
Therefore, $\prec_2 \; \subset \; \prec_1$.
If $\alpha \preccurlyeq_1 \beta$ and $\beta \preccurlyeq_1 \alpha$, then 
$\alpha \preccurlyeq_2 \beta$ and $\beta \preccurlyeq_2 \alpha$ again follows from the inclusion.
Therefore, $\sim_1 \; \subset \; \sim_2$.
Totality implies for $i = 1,2$
\begin{eqnarray}
\label{eq:4cgy}
\sim_i & = & \mathcal{M}\times\mathcal{M}\setminus (\prec_i \cup \succ_i) .
\end{eqnarray}
Hence, $\prec_1 \; = \; \prec_2$ would imply $\sim_1 \; = \; \sim_2$, and thus $\preccurlyeq_1 \; = \; \preccurlyeq_2$, 
which would be a contradiction.
Therefore, with the first part, $\prec_2 \; \subsetneq \; \prec_1$, and with \eqref{eq:4cgy}, $\sim_1 \; \subsetneq \; \sim_2$.
\end{proof}

\begin{corollary}
\label{cor:2xam}
Let $\preccurlyeq_1 \; \subsetneq \; \preccurlyeq_2$ be two total preorders. 
Then $\preccurlyeq \;  := \; \sim_1 \cup \prec_2$ is a preorder.
If $\preccurlyeq_1$ and $\preccurlyeq_2$ both
have the property \ref{anti}, \ref{add}, or \ref{PH} of Def.~\ref{def:relations}, 
then so does $\preccurlyeq$.
\end{corollary}

\begin{proof}
Prop.~\ref{prop:3oiu} and Prop.~\ref{prop:1kop}.
\end{proof}

\begin{definition}
\label{def:partial order}
A partial order
is an antisymmetric preorder (\ref{T}, \ref{ref}, and \ref{anti} of Def.~\ref{def:relations}).
\end{definition}

\begin{lemma}
\label{lem:8klm}
Let $\preccurlyeq$ be a partial order. Then $\alpha \sim \beta$ if and only if $\alpha = \beta$.
\end{lemma}

\begin{lemma}
The reflexive closure $(\prec')^r$ of a strict partial order $\prec'$ is a partial order and the smallest preorder that contains $\prec'$.
Hence, any strict partial order can be constructed by removing the reflexive pairs $\{(\alpha,\alpha) : \alpha\in\mathcal{M}\}$
from a partial order.
\end{lemma}


\section{Chain consistency}
\label{sec:Chain consistency}

Below, I introduce the for this work fundamental notion of consistency, or chain consistency, 
between two transitive relations.
Examples of consistency and inconsistency are given in this section, 
and I illustrate why the simple assumption
of disjoint symmetric and asymmetric parts across two transitive relations generally is not sufficient for 
their consistency.

\begin{definition}
\label{def:Consistency}
Two transitive relations $\preccurlyeq_1$ and $\preccurlyeq_2$  
are chain-consistent with one another if there exist no 
$(\gamma_i,\gamma_{i+1})\in (\preccurlyeq_1 \cup \preccurlyeq_2)$ for $i=1,\ldots,n-1$
and for some $n\in\mathbb{N}\setminus\{0,1\}$
such that $\gamma_1 = \gamma_n$ and
at least one $(\gamma_i,\gamma_{i+1})\in\;\prec_1 \cup \prec_2$.
\end{definition}

Chain consistency means
that there is no closed chain made up of $\sim_1$ and $\sim_2$, and at least one of $\prec_1$ or $\prec_2$.
For instance, some $\alpha \preccurlyeq_1 \beta \prec_2 \gamma \preccurlyeq_1 \delta \preccurlyeq_1 \alpha$ or some
$\alpha \sim_1 \beta \prec_2 \gamma \sim_1 \delta \sim_2 \epsilon \prec_1 \alpha$
would constitute a circular chain as in Def.~\ref{def:Consistency}. 
Chain consistency e.g.~prevents that $\alpha \preccurlyeq_1 \beta$ while $\beta \prec_2 \alpha$.

\begin{example}
Consider $\mathbb{R}\times\mathbb{R}$ with $(a_1,a_2) \sim' (b_1,b_2)$ 
if and only if $a_1 - 2a_2 = b_1 -2b_2$.
Moreover, $\prec''$ being defined as Pareto-better, i.e.~$(a_1,a_2) \prec'' (b_1,b_2)$ if and only if
either $a_1 \leq a_2$ or $b_1 < b_2$, or $a_1 < a_2$ and $b_1 \leq b_2$, or both. 
It is easy to check that $\prec''$ is a strict partial order and that $\sim'$ is an equivalence order.
Since $(1,2) \prec'' (5,4) \sim' (1,2)$, the relations $\sim'$ and $\prec''$ are not chain-consistent with one another.
\end{example}

\begin{lemma}
\label{lem:7pom}
For a transitive relation $\preccurlyeq$, the following are chain-consistent with one another:
(1) $\preccurlyeq$ and $\sim$, (2) $\preccurlyeq$ and $\prec$, as well as (3) $\sim$ and $\prec$.
Thus, if two transitive relations are chain-inconsistent, closed chains of the form specified in Def.~\ref{def:Consistency}
will have members from both relations.
\end{lemma}

\begin{proof}
For these pairs of relations, any closed chain as in Def.~\ref{def:Consistency} would by Lemma \ref{lem:1csv} imply $\gamma_1 \prec \gamma_1$,
thus contradicting the irreflexivity of $\prec$.  
\end{proof}

\begin{example}
\label{ex:1D}
For a linear space $\mathcal{M}$, consider a real linear map $f: \mathcal{V} \rightarrow \mathbb{R}$ on a subspace
$\mathcal{V} \subset \mathcal{M}$ which is not identical zero.
For $\alpha,\beta\in \mathcal{M}$, define $\alpha \preccurlyeq_f \beta$ whenever $f(\alpha) \leq f(\beta)$.
It is now easy to check that $\preccurlyeq_f$ is a vector preorder on $\mathcal{M}$,
and that $\alpha \sim_f \beta$ if and only if $\alpha-\beta\in \text{Ker}(f)$, and, thus, $\alpha \prec_f \beta$ if and only if $f(\alpha) < f(\beta)$.
By Lemma \ref{lem:7pom}, $\sim_f$ and $\prec_f$ are consistent with one another.
\end{example}

\begin{proposition}
\label{prop:ConsistencyII}
If $\preccurlyeq_1$ and $\preccurlyeq_2$ are chain-consistent with one another, then
\begin{eqnarray}
\label{eq:7uzn}
(\sim_1 \cup \sim_2) \; \cap \; (\prec_1 \cup \prec_2)
& = & 
\emptyset .
\end{eqnarray}
However, in general, \eqref{eq:7uzn} does not imply chain consistency of $\preccurlyeq_1$ and $\preccurlyeq_2$.
\end{proposition}

\begin{proof}
If \eqref{eq:7uzn} is violated, then there either exists $\alpha \sim_1 \beta$ with $\alpha \prec_2 \beta$, or
$\alpha \sim_2 \beta$ with $\alpha \prec_1 \beta$, both leading to chains excluded by Def.~\ref{def:Consistency}.
For the second part, consider distinct $\gamma_i$, $i=1,\ldots,4$, and
\begin{eqnarray}
\label{eq:6mkl}
\preccurlyeq_1
\;\; = \;\;
\sim_1
& := & 
\{(\gamma_i,\gamma_j): (i,j)\in\{ (1,4), (4,1), (2,3), (3,2), \\
\nonumber
& & \qquad\qquad\qquad\qquad\, (1,1), (2,2), (3,3), (4,4) \} \} , \\
\label{eq:7mkl}
\preccurlyeq_2
\;\; = \;\;
\prec_2
& := & 
\{(\gamma_1, \gamma_3), (\gamma_2, \gamma_4)\} . 
\end{eqnarray}
It is easy to check that the first is an equivalence relation and the second a strict partial order, and that 
\eqref{eq:7uzn} holds.
However, $\gamma_1 \prec_2 \gamma_3 \sim_1 \gamma_2 \prec_2 \gamma_4 \sim_1 \gamma_1$.
\end{proof}

So, while for a transitive relation $\preccurlyeq$ it holds that $\sim \cap \prec \; = \emptyset$, the reverse is not generally true, 
i.e.~if $\sim_1 \cap \prec_2 \; = \emptyset$ for an equivalence relation $\sim_1$ and a strict partial order $\prec_2$,
then $\sim_1 \cup \prec_2$ is not necessarily a transitive relation.

I now arrive at a first simple result.
Two transitive relations are chain-consistent with one another if and only if
their transitive closure $\preccurlyeq_{\preccurlyeq_1 \cup \, \preccurlyeq_2}$
does not contradict either of them.

\begin{lemma}
\label{lem:9oim}
Two transitive relations $\preccurlyeq_1$ and $\preccurlyeq_2$ are chain-consistent with one another 
if and only if there exist no $\alpha, \beta \in \mathcal{M}$ with
$\alpha \preccurlyeq_{\preccurlyeq_1 \cup \, \preccurlyeq_2} \beta$
such that 
either $\beta \prec_1 \alpha$, 
or $\beta \prec_2 \alpha$,
or both.
\end{lemma}

\begin{proof}
`$\Rightarrow$': By Def.~\ref{def:transitive closure} (recall Prop.~\ref{prop:transitive closure}), 
if the non-permitted situation of Lemma \ref{lem:9oim} prevails, then
a chain of $\preccurlyeq_1$ and $\preccurlyeq_2$ exists from $\alpha$ to $\beta$, and this chain can be 
closed with $\beta \prec_1 \alpha$ or $\beta \prec_2 \alpha$, such that there is no chain-consistency
according to Def.~\ref{def:Consistency}.

`$\Leftarrow$': If the relations are inconsistent with one another, then there exists a closed chain 
of $\preccurlyeq_1$ and $\preccurlyeq_2$
as not permitted by Def.~\ref{def:Consistency}, where at least one $(\gamma_i,\gamma_{i+1})\in\;\prec_1 \cup \prec_2$.
Setting $\beta = \gamma_i$ and $\alpha = \gamma_{i+1}$ and using Def.~\ref{def:transitive closure}, 
it is now clear that the non-permitted situation of Lemma \ref{lem:9oim}
is established. 
\end{proof}

Obviously, $\sim_1 \cup \sim_2 \, \subset \, \sim_{\preccurlyeq_1 \cup \, \preccurlyeq_2}$ holds
in the situation of Lemma \ref{lem:9oim}.
Moreover, the second condition in Lemma \ref{lem:9oim} can easily be shown to be equivalent to
$\prec_1 \cup \prec_2 \, \subset \, \prec_{\preccurlyeq_1 \cup \, \preccurlyeq_2}$.
This brings me to the definition of consistent extensions of transitive relations in the next section.


\section{Consistent extensions and completions}

This section provides the critical transfinite induction results needed in the proof of my main theorem in the next section.
The notion of a consistent extension and of a consistent completion of a transitive relation is introduced, and a 
technique to extend -- and, consequently, complete -- a non-total (non-complete) transitive relation in a 
consistent manner is explained. 

The following extension of a relation is structure preserving regarding the symmetric and asymmetric parts of the relation.

\begin{definition}
\label{def:12mkl}
Relation $\mathbf{R}_2$ is a consistent extension of relation $\mathbf{R}_1$,
written $\mathbf{R}_1 \sqsubset \mathbf{R}_2$, if $\rsym{R}_1 \subset \rsym{R}_2$ and $\rasy{R}_1 \subset \rasy{R}_2$.
\end{definition}

\begin{lemma}
\label{lem:9klw}
Let $\mathbf{S}$ be an arbitrary non-empty subset of $\mathcal{M}\times \mathcal{M}$. 
Then, $\preccurlyeq_\mathbf{S} \;  \sqsubset \; \preccurlyeq^r_\mathbf{S}$.

\begin{lemma}
\label{lem:8iuz}
If for two transitive relations $\preccurlyeq_1 \; \sqsubset \; \preccurlyeq_2$, then $\preccurlyeq_1$ and $\preccurlyeq_2$ are chain-consistent.
\end{lemma}

\begin{proof}
Because of the inclusions in Def.~\ref{def:12mkl}, chain inconsistency of $\preccurlyeq_1$ and $\preccurlyeq_2$ 
would imply that of $\sim_2$ and $\prec_2$, which is impossible by Lemma \ref{lem:7pom}.
\end{proof}

\begin{lemma}
If $\preccurlyeq_1 \; \sqsubset \; \preccurlyeq_2$ for a preorder $\preccurlyeq_1$ and a transitive relation $\preccurlyeq_2$,
then $\preccurlyeq_2$ is a preorder.
\end{lemma}

\end{lemma}

\begin{proposition}
\label{prop:Consistent order of transitive relations}
$\sqsubset$ is a partial order on all relations on $\mathcal{M}$ which share the same subset of properties from Def.~\ref{def:relations}.
Any totally $\sqsubset$-ordered subset $\mathcal{R}$ of these relations on $\mathcal{M}$ has the upper bound 
$\mathbf{B} = \bigcup_{\mathbf{R}\in \mathcal{R}} \mathbf{R}$,
which also has the properties. Moreover,
$\rsym{B} = \bigcup_{\mathbf{R}\in \mathcal{R}} \rsym{R}$
and
$\rasy{B} = \bigcup_{\mathbf{R}\in \mathcal{R}} \rasy{R}$.
\end{proposition}

\begin{proof}
Clearly, $\sqsubset$ is reflexive, antisymmetric, and transitive, and therefore a partial order on all
relations on $\mathcal{M}$ that share the same set of properties from Def.~\ref{def:relations}.
Clearly, $\mathbf{B} = (\bigcup_{\mathbf{R}\in \mathcal{R}} \rsym{R}) \cup (\bigcup_{\mathbf{R}\in \mathcal{R}} \rasy{R})$, 
and $\bigcup_{\mathbf{R}\in \mathcal{R}} \rsym{R}$ is symmetric. 
$\bigcup_{\mathbf{R}\in \mathcal{R}} \rasy{R}$ is asymmetric, because, by construction, it cannot have elements of the form
$(\alpha, \alpha)$, and if for $\alpha \neq \beta$ it held that $(\alpha, \beta), (\beta, \alpha)\in \bigcup_{\mathbf{R}\in \mathcal{R}} \rasy{R}$,
then, given the total ordering, there would have to be one member $\rasy{R}$ containing both of them, which cannot be.
Thus, $\rsym{B}$ and $\rasy{B}$ are as stated, and $\mathbf{R} \sqsubset \mathbf{B}$ for any $\mathbf{R}\in\mathcal{R}$.
Regarding the properties, consider first transitivity.
$\alpha \mathbf{R} \beta$ and $\beta \mathbf{R} \gamma$ means that
there exist $\mathbf{R}_1, \mathbf{R}_2\in\mathcal{R}$ such that $\alpha \mathbf{R}_1 \beta$ and $\beta \mathbf{R}_2 \gamma$.
If I assume without loss of generality that $\mathbf{R}_1 \sqsubset \mathbf{R}_2$, then also
$\alpha \mathbf{R}_2 \beta$, and therefore $\alpha \mathbf{R}_2 \gamma$ by transitivity of $\mathbf{R}_2$, 
which in turn implies $\alpha \mathbf{R} \gamma$. Thus, $\mathbf{R}$ is a transitive relation.
With similar or even simpler arguments, the remaining properties \ref{sym} to \ref{To}
of Def.~\ref{def:relations} can be checked to hold for $\mathbf{B}$ if they hold for each member of $\mathcal{R}$.
For instance, for property \ref{asy} (asymmetry), if all members of $\mathcal{R}$ have the property, then $\mathbf{B} = \rasy{B}$,
so $\mathbf{B}$ has the property, as well.
\end{proof}

In the following, $\preccurlyeq^{(add,ph)}$ means that $\preccurlyeq$ applies in cases where only transitivity is considered,
and $\preccurlyeq^{add,ph}$ applies in cases where, additionally, additivity and positively homogeneity are considered.

\begin{proposition}
\label{prop:Minimal consistent extensions of transitive relations}
For $\alpha \neq \beta$, $\alpha, \beta \in \mathcal{M}$, assume that 
neither $\alpha \preccurlyeq \beta$, nor $\beta \preccurlyeq \alpha$,
for a given transitive relation $\preccurlyeq$.
\begin{enumerate}
\item
If $\preccurlyeq$ is a transitive relation, preorder, or strict partial order on the set $\mathcal{M}$, then
\begin{eqnarray}
\label{eq:3zgh}
\preccurlyeq
& \sqsubset &
\preccurlyeq_{\preccurlyeq\cup\{(\alpha,\beta)\}} 
\;\;\,
=
\;\;
\left(
\preccurlyeq
\;
\cup
\;\,
\{(\gamma, \delta) : \gamma\preccurlyeq\alpha \text{ or } \gamma=\alpha , \text{ and } \beta\preccurlyeq\delta \text{ or } \beta = \delta\} \right) , \qquad
\end{eqnarray}
and the transitive closure $\preccurlyeq_{\preccurlyeq\cup\{(\alpha,\beta)\}}$ is again
a transitive relation, preorder, or strict partial order.
\item
If $\preccurlyeq$ is a vector preorder on the real linear space $\mathcal{M}$, then
\begin{eqnarray}
\label{eq:3zghB}
\preccurlyeq
& \sqsubset &
\preccurlyeq^{add,ph}_{\preccurlyeq\cup\{(\alpha,\beta)\}} 
\;\;\,
=
\;\;
\left(
\preccurlyeq
\;
\cup
\;\,
\{(r\gamma + \epsilon, r\delta+\zeta) : \gamma\preccurlyeq\alpha, \beta\preccurlyeq\delta, 
\epsilon\preccurlyeq\zeta, r\in\mathbb{R}_{>0}^+\} \right) , \qquad
\end{eqnarray}
where $\preccurlyeq^{add,ph}_{\preccurlyeq\cup\{(\alpha,\beta)\}}$ is a vector preorder.
\end{enumerate}
In both cases, $\preccurlyeq \;\, \neq \;\, \preccurlyeq^{(add,ph)}_{\preccurlyeq\cup\{(\alpha,\beta)\}}$, and
\begin{eqnarray}
\label{eq:4uzgh}
\alpha
& \prec^{(add,ph)}_{\preccurlyeq\cup\{(\alpha,\beta)\}} &
\beta ,
\end{eqnarray}
as well as
\begin{eqnarray}
\label{eq:4Auzgh}
\sim^{(add,ph)}_{\preccurlyeq\cup\{(\alpha,\beta)\}}
& = &
\sim .
\end{eqnarray}
\end{proposition}

This means that the transitive relation, preorder, or strict partial order $\preccurlyeq_{\preccurlyeq\cup\{(\alpha,\beta)\}}$ is 
the smallest one larger than 
$\preccurlyeq$ that also strictly prefers $\beta$ over $\alpha$. 
Moreover, only the strict part of $\preccurlyeq$ is extended by this procedure.
A result in the case of partial orders that is somewhat related to Proposition \ref{prop:Minimal consistent extensions of transitive relations}
is Theorem 1.19 in \cite{RoSt}.

\begin{proof}
Part 1.
Denote the right hand side of \eqref{eq:3zgh} short with $\preccurlyeq'$.
Note that $\alpha \preccurlyeq' \beta$.
Clearly, $\preccurlyeq \; \subset \; \preccurlyeq'$, and $\preccurlyeq'$ is reflexive in the preorder case.
For the second equality in \eqref{eq:3zgh}, it only needs to be shown that the relation on the right,
which is $\preccurlyeq'$, is transitive, since it obviously is contained
in $\preccurlyeq_{\preccurlyeq\cup\{(\alpha,\beta)\}}$.
For transitivity, assume $\gamma_1 \preccurlyeq' \gamma_2$ and $\gamma_2 \preccurlyeq' \gamma_3$.
Case 1: $\gamma_1 \preccurlyeq \gamma_2$ and $\gamma_2 \preccurlyeq \gamma_3$. Then, $\gamma_1 \preccurlyeq \gamma_3$, and thus
$\gamma_1 \preccurlyeq' \gamma_3$.
Case 2: $\gamma_1 \preccurlyeq \gamma_2$, but $\gamma_2 \not\preccurlyeq \gamma_3$. 
Then, $\gamma_1 \preccurlyeq \gamma_2 \preccurlyeq \alpha$ 
or $\gamma_1 \preccurlyeq \gamma_2 = \alpha$,
and $\beta \preccurlyeq \gamma_3$ or $\beta = \gamma_3$, 
and thus, because $\gamma_1 \preccurlyeq \alpha$,
$\gamma_1 \preccurlyeq' \gamma_3$. 
Case 3: $\gamma_1 \not\preccurlyeq \gamma_2$, but $\gamma_2 \preccurlyeq \gamma_3$. 
Then, $\gamma_1 \preccurlyeq \alpha$ or $\gamma_1 = \alpha$,
and $\beta \preccurlyeq \gamma_2 \preccurlyeq \gamma_3$
or $\beta = \gamma_2 \preccurlyeq \gamma_3$, and thus, because $\beta \preccurlyeq \gamma_3$,
$\gamma_1 \preccurlyeq' \gamma_3$.
Case 4: $\gamma_1 \not\preccurlyeq \gamma_2$ and $\gamma_2 \not\preccurlyeq \gamma_3$. 
Then, $\gamma_1 \preccurlyeq \alpha$ or $\gamma_1 = \alpha$,
and $\beta \preccurlyeq \gamma_2$ or $\beta = \gamma_2$, 
as well as $\gamma_2 \preccurlyeq \alpha$ or $\gamma_2 = \alpha$,
and $\beta \preccurlyeq \gamma_3$ or $\beta = \gamma_3$.
This implies $\beta \preccurlyeq \alpha$ or $\beta = \alpha$, which is a contradiction to the assumptions. 
Thus, this case cannot occur.
It is now established that the right hand side of \eqref{eq:3zgh} is a transitive relation,
and a preorder, if $\preccurlyeq$ is one.
Since $\beta \preccurlyeq' \alpha$ would with \eqref{eq:3zgh} either imply $\alpha = \beta$, or imply both, 
$\alpha \preccurlyeq \beta$ and $\beta \preccurlyeq \alpha$, which is a contradiction to the initial assumptions,
$\beta \not\preccurlyeq' \alpha$, and \eqref{eq:4uzgh} holds.
This, in turn, implies via Lemma \ref{lem:1csv} and
the right hand side of \eqref{eq:3zgh} that $\sim' \; = \; \sim$, and \eqref{eq:4Auzgh} holds.
Since $\preccurlyeq \; \subset \; \preccurlyeq'$, $\sim' \; = \; \sim$ implies 
$\prec \; \subset \; \prec'$, and, in conclusion, $\sqsubset$ holds on the left hand side in \eqref{eq:3zgh}.
However, with Lemma \ref{lem:Implied equivalence and strict preference},
this also implies that $\preccurlyeq'$ is a strict partial order if $\preccurlyeq$ is, since then $\sim' \; = \; \sim \; = \, \emptyset$.

Part 2.
Again, denote the right hand side of \eqref{eq:3zghB} short with $\preccurlyeq'$. Reflexivity of the second and third
relation in \eqref{eq:3zghB} is given by the reflexivity of $\preccurlyeq$.
For the second equality in \eqref{eq:3zghB}, it now only needs to be shown that the relation on the right is 
additive, and positively homogeneous, because reflexivity and additivity imply transitivity, and thus a preorder,
and it is easy to see that the relation on the right must be contained in $\preccurlyeq^{add,ph}_{\preccurlyeq\cup\{(\alpha,\beta)\}}$,
which, since in this case $\preccurlyeq$ is reflexive, is the transitive, reflexive, additive, and positively homogeneous relation generated by
$\preccurlyeq\cup\,\{(\alpha,\beta)\}$.
For additivity, first assume 
$r_1 \gamma_1 + \epsilon_1 \preccurlyeq'  r_1 \delta_1 + \zeta_1$
and
$r_2 \gamma_2 + \epsilon_2 \preccurlyeq'  r_2 \delta_2 + \zeta_2$
as in the right hand side of \eqref{eq:3zghB}.
Additivity and positive homogeneity of $\preccurlyeq$ imply
$(\epsilon_1 + \epsilon_2) \preccurlyeq (\zeta_1 + \zeta_2)$, 
$(r_1 \gamma_1 + r_2 \gamma_2)/(r_1+r_2) \preccurlyeq \alpha$, and $\beta \preccurlyeq (r_1 \delta_1 + r_2 \delta_2)/(r_1+r_2)$.
Therefore, \eqref{eq:3zghB} implies
\begin{eqnarray}
\epsilon_1 + \epsilon_2 
+
r_1 \gamma_1 + r_2 \gamma_2
& \preccurlyeq' &
\zeta_1 + \zeta_2
+
r_1 \delta_1 + r_2 \delta_2 .
\end{eqnarray}
Assume now
$r_1 \gamma + \epsilon_1 \preccurlyeq'  r_1 \delta + \zeta_1$
as in the right hand side of \eqref{eq:3zghB}, as well as 
$\epsilon_2 \preccurlyeq \zeta_2$.
Since $\epsilon_1 \preccurlyeq \zeta_1$, and therefore 
$\epsilon_1 + \epsilon_1 \preccurlyeq \zeta_1 + \zeta_2$, 
and since $\gamma \preccurlyeq \alpha$ and $\beta \preccurlyeq \delta$,
\eqref{eq:3zghB} implies
\begin{eqnarray}
\epsilon_1 + \epsilon_2 
+
r_1 \gamma
& \preccurlyeq' &
\zeta_1 + \zeta_2
+
r_1 \delta .
\end{eqnarray}
Therefore, additivity is established. 
The positive homogeneity of the right hand side of \eqref{eq:3zghB} is straightforward to see, thus,
$\preccurlyeq'$ is an additive positively homogeneous preorder.
Assume now $\beta \preccurlyeq' \alpha$. 
According to \eqref{eq:3zghB}, and since $\beta \not\preccurlyeq \alpha$, there exist 
$\gamma\preccurlyeq\alpha$, $\beta\preccurlyeq\delta$, $\epsilon\preccurlyeq\zeta$, and $r\in\mathbb{R}_{>0}^+$
such that
$\beta = r\gamma + \epsilon$
and 
$\alpha =  r\delta+\zeta$.
This implies
$\gamma = (\beta - \epsilon)/r$
and 
$\delta = (\alpha - \zeta)/r$.
Since $\gamma\preccurlyeq\alpha$ and $\beta\preccurlyeq\delta$, 
one obtains
$(\beta - \epsilon)/r\preccurlyeq\alpha$
and
$\beta\preccurlyeq(\alpha - \zeta)/r$.
Multiplying each with $r$ and adding them,
\begin{eqnarray}
(1+r)\beta - \epsilon
& \preccurlyeq &
(1+r)\alpha - \zeta .
\end{eqnarray}
Adding $\epsilon\preccurlyeq\zeta$ and multiplying with $1/(1+r)$, I obtain the contradiction 
$\beta \preccurlyeq \alpha$.
Thus, \eqref{eq:4uzgh} holds.
This, in turn, implies via Lemma \ref{lem:1csv} and
the right hand side of \eqref{eq:3zghB} that $\sim' \; = \; \sim$, and \eqref{eq:4Auzgh} holds.
\end{proof}

\begin{corollary}
\label{cor:4ppm}
If a transitive relation
$\preccurlyeq_1$ 
and a strict partial order 
$\prec_2$ 
are chain-consistent with one another, then
the implied equivalence relation
$\sim_1$ 
and
$\prec_2$ 
are chain-consistent with one another.
The reverse is not necessarily true.
\end{corollary}

\begin{proof}
If $\sim_1$ and $\prec_2$ are not chain-consistent, there is a circular chain as described in Def.~\ref{def:Consistency} for $\sim_1$ and $\prec_2$.
However, since $\sim_1 \; \subset \; \preccurlyeq_1$, this implies a circular chain as in Def.~\ref{def:Consistency} for $\preccurlyeq_1$ and $\prec_2$.
To see that the reverse is not true, use for a non-total equivalence relation, $\sim$,
Prop.~\ref{prop:Minimal consistent extensions of transitive relations}, \eqref{eq:3zgh}, with some 
$\alpha \nsim \beta$ to obtain 
\begin{eqnarray}
\label{eq:14tzb}
\preccurlyeq_1 
\quad := \quad
\preccurlyeq_{\sim \cup \{(\alpha,\beta)\}}
& = & 
\sim
\;
\cup
\;\,
\{(\gamma, \delta) : \gamma\sim\alpha \text{ or } \gamma=\alpha, \text{ and }  \beta\sim\delta \text{ or } \beta=\delta\} ,\\
\label{eq:15tzb}
\preccurlyeq_2
\quad := \quad
\preccurlyeq_{\sim \cup \{(\beta,\alpha)\}}
& = & 
\sim
\;
\cup
\;\,
\{(\gamma, \delta) : \gamma\sim\beta \text{ or } \gamma=\beta,  \text{ and } \alpha\sim\delta \text{ or } \alpha=\delta\} .
\end{eqnarray}
By Prop.~\ref{prop:Minimal consistent extensions of transitive relations}, Eq.~\eqref{eq:4Auzgh}, $\sim_1 \; = \; \sim_2 \; = \; \sim$,
which implies with Lemma \ref{lem:7pom} that $\sim_1$ is consistent with $\prec_2$, but $\alpha\preccurlyeq_1 \beta \prec_2 \alpha$, so $\preccurlyeq_1$
and $\prec_2$ are not chain-consistent with one another.
\end{proof}

\begin{definition}
\label{def:Completion of preorders}
For a relation $\mathbf{R}$ on $\mathcal{M}$, the relation $\ol{\mathbf{R}}$ on $\mathcal{M}$ is a consistent completion 
of $\mathbf{R}$ if $\ol{\mathbf{R}}$ is total and $\mathbf{R} \sqsubset \ol{\mathbf{R}}$.
\end{definition}

Note that any relation is trivially extended to a total one by $\mathcal{M}\times \mathcal{M}$.
However, structure-preserving, i.e.~consistent completion is less trivial. 
Recall from Lemma \ref{lem:1lkm} that a total transitive relation is a preorder.

\begin{theorem}
\label{theo:Consistent completion of transitive relations}
Any transitive relation has a consistent completion which is a preorder.
Any vector preorder has a consistent completion which again
is a vector preorder.
\end{theorem}

\begin{proof}
Let $\preccurlyeq$ be a non-total transitive relation or vector preorder on a set or real linear space $\mathcal{M}$. 
Consider now the set of all transitive relations or all vectors preorders
on $\mathcal{M}$ which are at least as large as $\preccurlyeq$ with respect to $\sqsubset$. 
Clearly, $\sqsubset$ is a partial order on these relations.
By Prop.~\ref{prop:Consistent order of transitive relations}, any totally ordered subset has an upper bound.
Therefore, by Zorn's Lemma (cf.~\cite{ZoMa}; the usage of Zorn's ``maximum principle'' here is 
is exactly as formulated by him, namely for inclusion orders on sets), 
these relations have at least one maximal element, of which one is chosen and denoted by $\ol{\preccurlyeq}$.
Assume now that for some $\alpha, \beta \in \mathcal{M}$ neither $\alpha \; \ol{\preccurlyeq} \; \beta$, nor $\beta \; \ol{\preccurlyeq} \; \alpha$. 
If $\alpha \neq \beta$, then, by Prop.~\ref{prop:Minimal consistent extensions of transitive relations},
$\preccurlyeq^{(add,ph)}_{\ol{\preccurlyeq}\cup\{(\alpha,\beta)\}}$ is strictly $\sqsubset$-larger than $\ol{\preccurlyeq}$, which is 
a contradiction to $\ol{\preccurlyeq}$ being a maximal element.
In the case of transitive non-reflexive relations, $\alpha = \beta$ is a possibility, too. Then it is straightforward to see that
$\ol{\preccurlyeq}\cup\{(\alpha,\alpha)\}$ is strictly $\sqsubset$-larger than $\ol{\preccurlyeq}$ and transitive, 
which, again, is a contradiction to $\ol{\preccurlyeq}$ being a maximal element. 
\end{proof}


\section{A consistency theorem for transitive relations}
\label{sec:Main result}

This section proves my main theorem and states two of its corollaries.

\begin{theorem}
\label{theo:2qwc}
Let $\preccurlyeq_1$ and $\preccurlyeq_2$ be transitive relations on a set $\mathcal{M}$.
\begin{enumerate}
\item
The following are equivalent:
\begin{enumerate}
\item
$\preccurlyeq_1$ and $\preccurlyeq_2$ are chain-consistent with one another.
\item
$\preccurlyeq_{\preccurlyeq_1 \cup \, \preccurlyeq_2}$ is chain-consistent with $\preccurlyeq_1$ and with $\preccurlyeq_2$.
\item
$\preccurlyeq_{\preccurlyeq_1 \cup \, \preccurlyeq_2}$ consistently extends $\preccurlyeq_1$ and $\preccurlyeq_2$.
\item
There exists a transitive relation that consistently extends $\preccurlyeq_1$ and $\preccurlyeq_2$.
\item
There exists a common transitive consistent completion for $\preccurlyeq_1$ and $\preccurlyeq_2$.
\item
There exists a transitive consistent completion $\ol{\preccurlyeq}_1$ of $\preccurlyeq_1$, 
and $\ol{\preccurlyeq}_1$ and $\preccurlyeq_2$ are chain-consistent with one another.
\end{enumerate}
\item
$\preccurlyeq_1$ and $\preccurlyeq_2$ are chain-consistent with one another and $\preccurlyeq^r_{\preccurlyeq_1 \cup \, \preccurlyeq_2}$ is
total if and only if there exists a uniquely determined common transitive consistent completion for $\preccurlyeq_1$ and $\preccurlyeq_2$.
In this case, the preorder $\preccurlyeq^r_{\preccurlyeq_1 \cup \, \preccurlyeq_2}$ is the consistent completion.
\item
If $\preccurlyeq_1$ and $\preccurlyeq_2$ are vector preorders
on the real linear space $\mathcal{M}$, the statements above hold while
$\preccurlyeq_{\preccurlyeq_1 \cup \, \preccurlyeq_2}$ and the extensions, completions, or total preorders 
again are vector preorders.
\end{enumerate}
\end{theorem}

\begin{proof}
Part 1. 
`(a) $\Rightarrow$ (b)': Assume there is a closed chain for 
$\preccurlyeq_{\preccurlyeq_1 \cup \, \preccurlyeq_2}$ and 
$\preccurlyeq_2$ as in Def.~\ref{def:Consistency}.
Since, by Prop.~\ref{prop:transitive closure}, $\preccurlyeq_{\preccurlyeq_1 \cup \, \preccurlyeq_2}$ is the transitive closure of 
$\preccurlyeq_1$ and $\preccurlyeq_2$, there is a contradiction to (a) if the chain contains at least one $\prec_2$. 
Furthermore, it is straightforward to see that $\sim_2 \; \subset \; \sim_{\preccurlyeq_1 \cup \, \preccurlyeq_2}$. 
Thus, the chain must lie entirely in $\preccurlyeq_{\preccurlyeq_1 \cup \, \preccurlyeq_2}$
with at least one link from $\prec_{\preccurlyeq_1 \cup \, \preccurlyeq_2}$,
but this contradicts Lemma \ref{lem:7pom} (3). 
A similar argument proves chain consistency with $\preccurlyeq_1$.

`(b) $\Rightarrow$ (c)': It is easy to see that $(\sim_1 \cup \sim_2) \subset \; \sim_{\preccurlyeq_1 \cup \, \preccurlyeq_2}$.
Assuming that (c) does not hold, $(\prec_1 \cup \prec_2) \not\subset \; \prec_{\preccurlyeq_1 \cup \, \preccurlyeq_2}$ follows. 
Then at least one of the two following cases applies.
Case 1: $\alpha \prec_2 \beta$ and $\alpha \sim_{\preccurlyeq_1 \cup \, \preccurlyeq_2} \beta$, where
the latter implies $\beta \preccurlyeq_{\preccurlyeq_1 \cup \, \preccurlyeq_2} \alpha$. 
Therefore,  $\alpha \prec_2 \beta \preccurlyeq_{\preccurlyeq_1 \cup \, \preccurlyeq_2} \alpha$, a contradiction to (b). 
Case 2: 
$\alpha \prec_1 \beta$ and $\alpha \sim_{\preccurlyeq_1 \cup \, \preccurlyeq_2} \beta$, where
the latter implies $\beta \preccurlyeq_{\preccurlyeq_1 \cup \, \preccurlyeq_2} \alpha$. 
Therefore,  $\alpha \prec_1 \beta \preccurlyeq_{\preccurlyeq_1 \cup \, \preccurlyeq_2} \alpha$, again a contradiction to (b).

`(c) $\Rightarrow$ (d)': Trivial.
`(d) $\Rightarrow$ (e)': Theorem \ref{theo:Consistent completion of transitive relations}.
`(e) $\Rightarrow$ (f)': Lemma \ref{lem:8iuz}.

`(f) $\Rightarrow$ (a)': Assume there is a closed chain for $\preccurlyeq_1$ and $\preccurlyeq_2$ as in Def.~\ref{def:Consistency}.
Because $\preccurlyeq_1 \sqsubset \ol{\preccurlyeq}_1$, 
this means
that there is a closed chain as in Def.~\ref{def:Consistency} with members from both,  $\ol{\preccurlyeq}_1$ and $\preccurlyeq_2$,
which is a contradiction to (f). 

Part 2. `$\Rightarrow$':
With part 1 and with Lemma \ref{lem:9klw}, 
$\preccurlyeq \; := \; \preccurlyeq^r_{\preccurlyeq_1 \cup \preccurlyeq_2}$ consistently completes $\preccurlyeq_1$ and $\preccurlyeq_2$.
Assume now that $\preccurlyeq^\circ$ is a reflexive transitive relation (i.e.~a preorder) 
different from $\preccurlyeq$, for which --  like for $\preccurlyeq$ -- also
$\preccurlyeq_1 \; \sqsubset \; \preccurlyeq^\circ$ and $\preccurlyeq_2 \; \sqsubset \; \preccurlyeq^\circ$.
Since $\preccurlyeq$ is inclusion-minimal (cf.~Def.~\ref{def:Generated relation} and Cor.~\ref{cor:generatedrelation}) 
with these properties, $\preccurlyeq \; \subset \; \preccurlyeq^\circ$ holds.
Because of this inclusion, $\preccurlyeq^\circ$ must be total. 
Moreover, since $\preccurlyeq^\circ$ is supposed to be different from $\preccurlyeq$, one has 
$\preccurlyeq \; \subsetneq \; \preccurlyeq^\circ$.
By Cor.~\ref{cor:2xam} and by Prop.~\ref{prop:3oiu}, $\preccurlyeq_3 \;  := \; \sim \cup \prec^\circ$ now is a preorder which is strictly inclusion-smaller 
than $\preccurlyeq$. 
Furthermore $\preccurlyeq_1 \; \sqsubset \; \preccurlyeq_3$ and $\preccurlyeq_2 \; \sqsubset \; \preccurlyeq_3$. 
This is a contradiction to the minimality of $\preccurlyeq$.

`$\Leftarrow$': 
Chain consistency of $\preccurlyeq_1$ and $\preccurlyeq_2$ is clear from part 1.
It also follows from part 1 and from Lemma \ref{lem:9klw} that 
the preorder $\preccurlyeq \; := \; \preccurlyeq^r_{\preccurlyeq_1 \cup \preccurlyeq_2}$
consistently extends $\preccurlyeq_1$ and $\preccurlyeq_2$.
If $\preccurlyeq$ is non-total, and therefore
for some $\alpha, \beta \in \mathcal{M}$ with $\alpha \neq \beta$
neither $\alpha \preccurlyeq \beta$, nor $\beta \preccurlyeq \alpha$,
Prop.~\ref{prop:Minimal consistent extensions of transitive relations} 
shows that
\begin{eqnarray}
\preccurlyeq
& \sqsubset &
\preccurlyeq_{\preccurlyeq\cup\{(\alpha,\beta)\}} ,   \\
\preccurlyeq
& \sqsubset &
\preccurlyeq_{\preccurlyeq\cup\{(\beta, \alpha)\}} ,  \\
\preccurlyeq_{\preccurlyeq\cup\{(\beta, \alpha)\}}
& \neq &
\preccurlyeq_{\preccurlyeq\cup\{\alpha,\beta\}} ,  \\
(\alpha,\beta)
& \subset  &
\prec_{\preccurlyeq\cup\{(\alpha,\beta)\}} , \\
(\beta,\alpha)
& \subset  &
\prec_{\preccurlyeq\cup\{(\beta,\alpha)\}}  .
\end{eqnarray}
$\preccurlyeq_{\preccurlyeq\cup\{(\alpha,\beta)\}}$ and $\preccurlyeq_{\preccurlyeq\cup\{(\beta, \alpha)\}}$
can be consistently completed to transitive relations by Theorem \ref{theo:Consistent completion of transitive relations}. 
Since $(\alpha,\beta)$ will be in the asymmetric part of one of the completions, and $(\beta, \alpha)$
will be in the corresponding part of the other one, 
this proves that the {\em a priori} assumed common completion is not unique in this case.
Therefore, uniqueness of the completion implies totality of $\preccurlyeq$.

Part 3.
For part 1 this holds because of Lemma \ref{lem:3mjy} and Theorem \ref{theo:Consistent completion of transitive relations}.
For part 2,  Lemma \ref{lem:3mjy}, Cor.~\ref{cor:2xam}, Prop.~\ref{prop:3oiu}, 
Prop.~\ref{prop:Minimal consistent extensions of transitive relations}, and Theorem \ref{theo:Consistent completion of transitive relations}
cover the additive positively homogeneous case.
\end{proof}

The equivalence of 1 (a) and 1 (c) of Theorem \ref{theo:2qwc} was already implied by
Lemma \ref{lem:9oim}, the remarks underneath its proof, and by Def.~\ref{def:12mkl}. Beside the vector preorder case and
other statements, the theorem now, of course, added the completeness and uniqueness results.

\begin{corollary}
If $\preccurlyeq_1$ and $\preccurlyeq_2$ are chain-consistent with one another and $\preccurlyeq^r_{\preccurlyeq_1 \cup \preccurlyeq_2}$ is total, 
then $\preccurlyeq^r_{\preccurlyeq_1 \cup \preccurlyeq_2}$ is the only preorder $\preccurlyeq$ fulfilling 
$\preccurlyeq_1 \; \sqsubset \; \preccurlyeq$ and $\preccurlyeq_2 \; \sqsubset \; \preccurlyeq$.
\end{corollary}

Next, I finally provide an explicit statement of the meaningful reverse of the dissection of any preorder
into an equivalence relation and a strict partial order that I mentioned in the first section and that was
formalized by Lemma \ref{lem:Implied equivalence and strict preference}.
If the equivalence relation and the strict partial order are chain-consistent with one another, 
then their transitive closure consistently extends both of them, meaning: Old equivalence then implies new (extended)
equivalence, and old strict ordering implies new (extended) strict ordering. Furthermore, this transitive closure 
also can be completed while keeping the just explained implications in place.

The following formulation of my result also seems to be related to the Fundamental Theorem of Asset Pricing, 
which sits at the heart of financial mathematics.
The analogies will be explored in Section \ref{sec:Application in financial economics}.

\begin{corollary}
\label{cor:6zrn}
Let $\sim'$ be an equivalence relation and let $\prec''$ be a strict partial order on a set $\mathcal{M}$.
\begin{enumerate}
\item
The following are equivalent:
\begin{enumerate}
\item
$\sim'$ and $\prec''$ are chain-consistent with one another.
\item
$\preccurlyeq_{\sim' \cup \prec''}$ and $\prec''$ are chain-consistent with one another.
\item
For the preorder $\preccurlyeq_{\sim' \cup \prec''}$ generated by $\sim'$ and $\prec''$,
\begin{eqnarray}
\label{eq:8werA}
\sim' & \subset & \sim_{\sim' \cup \prec''}  \quad \text{and}\\
\label{eq:9werA}
\prec'' & \subset & \prec_{\sim' \cup \prec''} .
\end{eqnarray}
\item
There exists a total preorder $\ol{\preccurlyeq}$ on $\mathcal{M}$ such that
\begin{eqnarray}
\label{eq:8wer}
\sim' & \subset & \ol{\sim} \quad \text{and}\\
\label{eq:9wer}
\prec'' & \subset & \ol{\prec} .
\end{eqnarray}
\end{enumerate}
\item
$\sim'$ and $\prec''$ are chain-consistent with one another and $\preccurlyeq_{\sim' \cup \prec''}$ is total 
if and only if there exists a uniquely determined total preorder $\ol{\preccurlyeq}$ on $\mathcal{M}$ with \eqref{eq:8wer} and \eqref{eq:9wer}.
In this case, $\ol{\preccurlyeq} = \; \preccurlyeq_{\sim' \cup \prec''}$.
\item
If $\sim'$ and $\prec''$ are chain-consistent and $\preccurlyeq_{\sim' \cup \prec''}$ is total, 
then $\preccurlyeq_{\sim' \cup \prec''}$ is the only preorder $\preccurlyeq$ fulfilling
\begin{eqnarray}
\label{eq:8werAC}
\sim' & \subset & \sim \quad \text{and}\\
\label{eq:9werAC}
\prec'' & \subset & \prec .
\end{eqnarray}
\end{enumerate}
\end{corollary}

\begin{proof}
Theorem \ref{theo:2qwc} for $\preccurlyeq_1 \; = \; \sim'$ and $\preccurlyeq_2 \; = \; \prec''$.
Note that $\preccurlyeq_{\sim' \cup \prec''}$ is trivially chain-consistent with $\sim'$.
Also note that $\preccurlyeq_{\sim' \cup \prec''}$ is reflexive, because $\sim'$ as an equivalence relation
is already.
\end{proof}


\section{Application in geometry: Cones}
\label{sec:Geometric Interpretation}

Let $\preccurlyeq$ be a vector preorder on the real linear space $\mathcal{M}$.  
Let $\mathcal{C}\subset\mathcal{M}$ be a positive cone with vertex $o$ (the nullvector), i.e.~for any $\alpha, \beta \in \mathcal{C}$
and $a,b \in \mathbb{R}_0^+$, $a\alpha + b\beta \in \mathcal{C}$.
In the literature, a positive cone as defined here is sometimes called convex, non-negative, or simply a `cone'.

\begin{definition}
\label{def:16tze}
\begin{eqnarray}
\label{eq:30erb}
\mathcal{C}_\preccurlyeq 
& = & 
\{\gamma \in \mathcal{M} : o \preccurlyeq \gamma \} 
\end{eqnarray}
and
\begin{eqnarray}
\label{eq:31erb}
\alpha \preccurlyeq_\mathcal{C} \beta
& \text{if and only if} &
\beta-\alpha \in \mathcal{C} .
\end{eqnarray}
\end{definition}
The following statement is well-known (e.g.~\cite{NaLa} or \cite{ScHe}).

\begin{proposition}
\label{prop:12oip}
$\mathcal{C}_\preccurlyeq$ is a positive cone in $\mathcal{M}$
and $\preccurlyeq_\mathcal{C}$ is a vector preorder on $\mathcal{M}$.
Furthermore,
$\mathcal{C}_\preccurlyeq = \mathcal{C}$ if and only if $\preccurlyeq \; = \; \preccurlyeq_\mathcal{C}$.
\end{proposition}

Prop.~\ref{prop:12oip} establishes a bijection between the vector preorders
on $\mathcal{M}$ and the positive cones in $\mathcal{M}$.

It is well-known (e.g.~\cite{NaLa}) that $\mathcal{C} \cap (-\mathcal{C})$ is the largest linear space in $\mathcal{C}$.
By \eqref{eq:31erb},
\begin{eqnarray}
\label{eq:3wqj}
o \prec_\mathcal{C} \gamma
& \text{if and only if} &
\gamma \in \mathcal{C} \setminus (-\mathcal{C}) .
\end{eqnarray}
Therefore, the following definition seems appropriate.

\begin{definition}
\begin{eqnarray}
\label{eq:33lob}
\mathcal{C}^{L} 
& = &
\mathcal{C} \cap (-\mathcal{C})
\end{eqnarray}
is called the linear part of $\mathcal{C}$ and 
\begin{eqnarray}
\label{eq:33lobB}
\mathcal{C}^{+} 
& = &
\mathcal{C} \setminus (-\mathcal{C})
\end{eqnarray}
is called the non-linear part or strictly positive part of $\mathcal{C}$.
\end{definition}

\begin{lemma}
\label{lem:14lpmZ}
For any $\alpha \in \mathcal{C}^{+}$, $\beta \in \mathcal{C}$ and $a,b>0$, 
it holds that $a\alpha + b\beta \in \mathcal{C}^{+}$. Moreover,
\begin{eqnarray}
\label{eq:36ztnA}
\alpha \sim_\mathcal{C} \beta
& \text{if and only if} &
\beta-\alpha \in \mathcal{C}^{L} ,\\
\label{eq:36ztn}
\alpha \prec_\mathcal{C} \beta
& \text{if and only if} &
\beta-\alpha \in \mathcal{C}^{+} .
\end{eqnarray}
\end{lemma}

\begin{proof}
$a\alpha + b\beta \in \mathcal{C}$ holds. If also $a\alpha + b\beta \in (-\mathcal{C})$, then, since $-b\beta \in (-\mathcal{C})$ and since 
$(-\mathcal{C})$ is a positive cone as well, 
it follows that $a\alpha, \alpha \in (-\mathcal{C})$, which is a contradiction.
\eqref{eq:36ztnA} follows directly from \eqref{eq:31erb} and \eqref{eq:33lob},
and \eqref{eq:36ztn} is \eqref{eq:3wqj}.
\end{proof}

Together with Lemma \ref{lem:8klm}, these observations lead to the following result.

\begin{corollary}
\label{cor:12oip}
There exists a bijection between the additive positively homogeneous partial orders
(i.e.~vector partial orders) on $\mathcal{M}$ and
the positive cones in $\mathcal{M}$ for which $\mathcal{C}^L = \{o\}$.
\end{corollary}

The 1-to-1 correspondence of Cor.~\ref{cor:12oip} is a restriction of the one from Prop.~\ref{prop:12oip}.

\begin{definition}
If $\mathcal{C} \cup (-\mathcal{C}) = \mathcal{M}$, I call $\mathcal{C}$ complete or total,
\end{definition}

It is easy to show that this condition holds if and only if $\preccurlyeq_\mathcal{C}$ is total.

\begin{corollary}
\label{cor:1pom}
If $\gamma_i \in \mathcal{C}$, $i=1,\ldots,n$, with at least one $k\in\{1,\ldots,n\}$ where $\gamma_k \in \mathcal{C}^{+}$, 
then $\sum_{i=1}^n \gamma_i \in \mathcal{C}^{+}$.
\end{corollary}

\begin{proof}
By induction from Lemma \ref{lem:14lpmZ}. 
\end{proof}

Expressed differently, Corollary \ref{cor:1pom} states that a walk, where each step is a member of the cone and at least one step is in the
strictly positive direction of the cone, cannot return to the same place and, in fact, ends up at a new place in strictly positive direction
from the old position.

Using again the analogy of a walk, what if one takes each step from one of two different cones,
$\mathcal{C}_1$ and $\mathcal{C}_2$, and makes sure at least one step is from the strictly positive part of one of the cones
- is it then possible to return to the same spot?
Theorem \ref{theo:2qwc} in its version for vector preorders
gives the answer that never being able to return is equivalent to that both cones can be embedded into
a larger cone such that both linear parts are embedded into the linear part of the larger cone, 
and both strictly positive parts are embedded into the strictly positive part.
To properly formulate this statement, I need the following definitions.

\begin{definition}
\label{def:19uhn}
Two positive cones $\mathcal{C}_1$ and $\mathcal{C}_2$ in a real linear space $\mathcal{M}$
are path-consistent with one another if there exist 
no $\delta_i \in \mathcal{C}_1\cup\mathcal{C}_2$, $i=1,\ldots,m$, with at least one $k\in\{1,\ldots,m\}$ where 
$\delta_k \in \mathcal{C}^{+}_1\cup\mathcal{C}^{+}_2$ such that $\sum_{i=1}^{m} \delta_i = o$.
\end{definition}

\begin{definition}
A positive cone $\mathcal{C}_2$ consistently extends a positive cone $\mathcal{C}_1$ 
 in a real linear space $\mathcal{M}$ if
$\mathcal{C}^{L}_1 \subset \mathcal{C}^{L}_2$ 
and
$\mathcal{C}^{+}_1 \subset \mathcal{C}^{+}_2$.
If $\mathcal{C}_2$ is complete, then it is called a consistent completion of $\mathcal{C}_1$.
\end{definition}

\begin{definition}
For two subsets $A$ and $B$ of a real linear space $\mathcal{M}$, I define
\begin{equation}
A \uplus B 
\; = \; 
A \cup (A + B) \cup B .
\end{equation}
\end{definition}

Obviously, $\uplus$ only differs from $+$ if $o\notin A$ or $o\notin B$.

\begin{corollary}
\label{cor:9uzt}
Given two positive cones $\mathcal{C}_1$ and $\mathcal{C}_2$ in a real linear space $\mathcal{M}$,
the following are equivalent.
\begin{enumerate}
\item
$\mathcal{C}_1$ and $\mathcal{C}_2$ are path-consistent with one another.
\item
There exist no $\delta_i \in \mathcal{C}_1 + \mathcal{C}_2$, $i=1,\ldots,m$, with at least one $k\in\{1,\ldots,m\}$ where 
$\delta_k \in \mathcal{C}^{+}_1 \uplus \mathcal{C}^{+}_2$ such that $\sum_{i=1}^{m} \delta_i = o$.
\item
There exists a cone $\mathcal{C} \subset \mathcal{M}$ that consistently extends $\mathcal{C}_1$ and $\mathcal{C}_2$.
$\mathcal{C}$ can be chosen to be a consistent completion. 
\item
There exists a cone $\mathcal{C} \subset \mathcal{M}$ such that
\begin{eqnarray}
\label{eq:37iopA}
\mathcal{C}^{L}_1 + \mathcal{C}^{L}_2 
&\subset&
\mathcal{C}^{L} \quad \text{resp.}\\
\label{eq:38iopA}
\mathcal{C}^{+}_1 \uplus \mathcal{C}^{+}_2 
&\subset&
\mathcal{C}^{+} .
\end{eqnarray}
$\mathcal{C}$ can be chosen to be complete. 
\item
The vector preorders $\preccurlyeq_{\mathcal{C}_1}$ and  $\preccurlyeq_{\mathcal{C}_2}$ 
are chain-consistent with one another.
\end{enumerate}
\end{corollary}

Thus, two positive cones are path-consistent with one another if and only if there exists 
a common consistent completion of them.

\begin{proof}
`1 $\Rightarrow$ 2':
If a set of $\delta_i$ as in statement 2 did exist (as opposed to not existing, as required), then splitting 
each $\delta_i$ into its constituents from $\mathcal{C}_1$ and from $\mathcal{C}_2$
or, if necessary, into its constituents from $\mathcal{C}^{+}_1$ and from $\mathcal{C}^{+}_2$
would create a set in direct contradiction to statement 1.

`1 $\Leftarrow$ 2': 
This is straightforward since 
$\mathcal{C}_1\cup\mathcal{C}_2 \subset \mathcal{C}_1 + \mathcal{C}_2$
and
$\mathcal{C}^{+}_1\cup\mathcal{C}^{+}_2 \subset \mathcal{C}^{+}_1 \uplus \mathcal{C}^{+}_2$.

`1 $\Rightarrow$ 5':
Because of Lemma \ref{lem:14lpmZ},
it is easy to check that, for $\preccurlyeq_{\mathcal{C}_1}$ and  $\preccurlyeq_{\mathcal{C}_2}$,
a chain with members $(\gamma_i,\gamma_{i+1})$ as in Def.~\ref{def:Consistency} would via
$\delta_i = \gamma_{i+1}-\gamma_i$ for $i=1,\ldots,n-1$ and $m = n-1$ imply 
$\{\delta_1,\ldots,\delta_{m}\}$ as in Def.~\ref{def:19uhn}.

`1 $\Leftarrow$ 5':
Vice versa, because of Lemma \ref{lem:14lpmZ},
$\{\delta_1,\ldots,\delta_{m}\}$ as in Def.~\ref{def:19uhn} would via 
$\gamma_1 = o$, $\gamma_i = \sum_{k=1}^{i-1} \delta_i$ for $i=2,\ldots,m+1$ and $n = m+1$
imply a chain as in Def.~\ref{def:Consistency} for $\preccurlyeq_{\mathcal{C}_1}$ and  $\preccurlyeq_{\mathcal{C}_2}$, 
which starts and ends in $o$. 

`3 $\Rightarrow$ 5':
The third statement implies
\begin{eqnarray}
\label{eq:37iop}
\mathcal{C}^{L}_1 \cup \mathcal{C}^{L}_2 
&\subset&
\mathcal{C}^{L} \quad \text{and}\\
\label{eq:38iop}
\mathcal{C}^{+}_1 \cup \mathcal{C}^{+}_2 
&\subset&
\mathcal{C}^{+} ,
\end{eqnarray}
and therefore implies with Lemma \ref{lem:14lpmZ} that the vector preorder $\preccurlyeq_{\mathcal{C}}$
consistently extends the vector preorders
$\preccurlyeq_{\mathcal{C}_1}$ and  $\preccurlyeq_{\mathcal{C}_2}$.

`3 $\Leftarrow$ 5':
Vice versa, if the vector preorder $\preccurlyeq_{\mathcal{C}}$ consistently extends the vector preorders
$\preccurlyeq_{\mathcal{C}_1}$ and  $\preccurlyeq_{\mathcal{C}_2}$, then Lemma \ref{lem:14lpmZ} implies
\eqref{eq:37iop} and \eqref{eq:38iop}.
Regarding completeness in statement 3, the equivalence of the completeness of $\mathcal{C}$ and $\preccurlyeq_{\mathcal{C}}$ 
had already been mentioned.
Thus, via the corresponding completeness statement of Theorem \ref{theo:2qwc}, equivalence of statement 1 and 3 is given.

`3 $\Rightarrow$ 4': 
Follows directly from the additivity of $\mathcal{C}^{L}$ and $\mathcal{C}^{+}$.

`3 $\Leftarrow$ 4': 
Follows directly from 
$\mathcal{C}_1\cup\mathcal{C}_2 \subset \mathcal{C}_1 + \mathcal{C}_2$
and
$\mathcal{C}^{+}_1\cup\mathcal{C}^{+}_2 \subset \mathcal{C}^{+}_1 \uplus \mathcal{C}^{+}_2$.
\end{proof}

Before a closing example for this section, 
and in analogy to Prop.~\ref{prop:ConsistencyII}, 
note that while path consistency of two cones via Corollary \ref{cor:9uzt} implies
\begin{eqnarray}
\label{eq:42tgv}
(\mathcal{C}^{L}_1 \cup \mathcal{C}^{L}_2)
\cap
(\mathcal{C}^{+}_1 \cup \mathcal{C}^{+}_2)
& = &
\emptyset ,
\end{eqnarray}
the reverse generally does not hold. A simple example for this in $\mathcal{M} = \mathbb{R}^2$ is given by 
$\mathcal{C}_1 = \mathbb{R} \times \mathbb{R}_0^+$ and $\mathcal{C}_2 = \mathbb{R} \times \mathbb{R}_0^-$.

\begin{example}
\label{ex:3qwv}
Expanding on Example \ref{ex:1D} with the vector preorder $\preccurlyeq_f$ on $\mathcal{M}$
stemming from a linear functional on 
a subspace $\mathcal{V} \subset \mathcal{M}$, one obtains from \eqref{eq:30erb} and \eqref{eq:33lob} that 
\begin{eqnarray}
\label{eq:9qyp}
\mathcal{C}_{\preccurlyeq_f} 
& = & 
f^{-1}(\mathbb{R}^+_0) ,\\
\label{eq:47ztbA}
\mathcal{C}^L_{\preccurlyeq_f} 
& = & 
f^{-1}(0) \; = \; \textnormal{Ker}(f) , \quad \text{and}\\
\label{eq:11qyp}
\mathcal{C}^{+}_{\preccurlyeq_f} 
& = & 
f^{-1}(\mathbb{R}^+_{>0}) , 
\end{eqnarray}
where the last one could be empty.
For a positive cone $\mathcal{C} \subset \mathcal{M}$,
assume now that there exist no $\delta_i\in f^{-1}(\mathbb{R}^+_0) \cup \,\mathcal{C}$, $i=1,\ldots,m$, 
with at least one $k\in\{1,\ldots,m\}$ where
$\delta_k \in f^{-1}(\mathbb{R}^+_{>0}) \, \cup \, \mathcal{C}^+$
such that $\sum_{i=1}^{m} \delta_i = o$. 
This means that $\mathcal{C}$ and the positive cone $\mathcal{C}_{\preccurlyeq_f}$ implied by $f$ are path-consistent.
It now holds that
\begin{eqnarray}
\label{eq:46ztb}
\mathcal{V} \cap \mathcal{C}
& \subset & 
f^{-1}(\mathbb{R}^+_0) , \\
\label{eq:47ztb}
\mathcal{V} \cap \mathcal{C}^L
& \subset & 
f^{-1}(0) , \quad \text{and}\\
\label{eq:48ztb}
\mathcal{V} \cap \mathcal{C}^{+}
& \subset & 
f^{-1}(\mathbb{R}^+_{>0}) .
\end{eqnarray}
To see \eqref{eq:46ztb}, assume that $f(\alpha) < 0$ for $\alpha\in \mathcal{V} \cap \mathcal{C}$. Then,
$\alpha - \alpha = o$, but $\alpha\in \mathcal{C}$ and $-\alpha\in f^{-1}(\mathbb{R}^+_{>0})$, which is a contradiction to
the path consistency requirement. \eqref{eq:47ztb} and \eqref{eq:48ztb} now follow with \eqref{eq:47ztbA} and \eqref{eq:11qyp}
from the consistent common extension statement of Corollary \ref{cor:9uzt}.

\eqref{eq:47ztb} with \eqref{eq:48ztb} is a stricter condition than the sometimes considered mere `positivity', 
which usually is understood as non-negativity, of $f$ on the cone $\mathcal{C}$ (e.g.~\cite{CaRe}). 
As such, the consistency between $\mathcal{C}$ and $f$, which is given as the path-consistency
between $\mathcal{C}$ and $\mathcal{C}_{\preccurlyeq_f}$, does not only imply a total cone which consistently extends
$\mathcal{C}$ and $\mathcal{C}_{\preccurlyeq_f}$, but it also implies a strict type of positivity of $f$ on $\mathcal{C}$.
\end{example}

In Example \ref{ex:3qwv},
if there existed a linear extension $F$ of $f$ onto $\mathcal{M}$ which implied the extending cone, 
this would be an addition of precision to the M.~Riesz Extension Theorem (e.g.~\cite{CaRe}).
However, I was not able to generally derive such a result.


\section{Application in financial economics}
\label{sec:Application in financial economics}

In financial economics or mathematical finance, an often encountered model setup is a market of traded financial goods -- 
meaning of deterministic or stochastic cash flows.
Typically, exchangeability is indirectly described by the property of having the same price, 
where price is the multiple of the num\'eraire, or unit of account, against which a good can be traded,
while it is assumed that all {\em a priori} given goods of the given market can be traded for a certain amount of the num\'eraire
(for the more mathematical articles on this topic, see \cite{DeFr} and the references therein).
Of course, any good can be exchanged with itself.
While usually not stated explicitly and because, typically, the price functional is {\em a priori} assumed to be linear 
in the goods, the relation of spot trades at time zero established by these assumptions
usually creates a linear equivalence relation on the goods, which are typically given in the form of a real linear space.
However, below, I will not require any linearity assumptions.

Another part of the model setups is that there is 
an understanding of what makes one financial good, or cash flow, for all market participants -- and, thus, objectively -- 
strictly better than another.
For instance, for stochastic cash flows, one cash flow usually is considered strictly better than another one 
if in all economic scenarios, meaning in any combination of a time and of an event with a non-zero probability, it pays
at least as much as the corresponding other cash flow, and in at least one such scenario it pays strictly more.

In economic theory, preorders are used as preference orders which describe the economic preferences 
of market participants regarding certain goods.
Strict partial orders are used as as strict preference orders, and equivalence relations emerge as indifference relations
(see e.g.~\cite{MaAn}, a widely known and appreciated textbook).

Summarily, it can therefore be assumed that there are two transitive relations which describe the typical
finance models outlined above.
First, an objective indifference relation, 
i.e.~an equivalence relation, $\sim'$, can be used to describe which goods can be exchanged for one another in this market.
Any market participant can be indifferent towards two goods $\alpha \sim' \beta$, because she or he can swap one for the
other at any time. Note that I dropped any linearity assumption for greater generality.
Second, there separately exists an objective strict preference order, $\prec''$, in accordance to which all market participants act.

\begin{definition}
A financial market is given by a set of financial goods or cash flows, $\mathcal{M}$, together with
an objective indifference relation, $\sim'$, on $\mathcal{M}$, which is an equivalence relation,
and an objective strict preference order, $\prec''$, on $\mathcal{M}$,
which is a strict partial order. 
\end{definition}

Two types of trades are possible for an agent in this market: 
\begin{enumerate}
\item Exchange a cash flow for one which is 
considered equivalent to the first one by the market, i.e.~receive $\alpha$ for $\beta$ if $\alpha \sim' \beta$.
Therefore, $\sim'$ is the `fair exchange' relation.
\item Trade a cash flow in for one which is considered strictly worse than the first one by the market, 
i.e.~receive $\alpha$ for $\beta$ if $\alpha \prec'' \beta$. Thus, $\prec''$ can be viewed as the `down-trade' relation.
\end{enumerate}
The following definition for the transitive closure of the two relations makes sense to me.
Note that any transitive closure of a reflexive relation naturally is reflexive, too.

\begin{definition}
\label{def:24wec}
The preorder $\preccurlyeq_{\sim' \cup \prec''}$ generated by the fair exchange relation (objective indifference relation) 
and by the down-trade relation (objective strict preference order) is called the trade relation or the attainable trades.
\end{definition}

An important question in finance now is under which conditions a given market is rational in the sense that
there exist no so-called arbitrage opportunities. 
These are trades, or chains of trades, where a cash flow can be swapped for an objectively better one at no additional cost.
Such trades are also called a `free lunch' (cf.~\cite{DeFr}), or it is said that someone made `money out of nothing'.
Simply put, these deals are `too good to be true'.

\begin{definition}
An arbitrage opportunity is a finite chain of either fair exchanges, down-trades, or both, which summarily
exchanges a good $\alpha\in\mathcal{M}$ for an objectively better one, $\beta\in\mathcal{M}$, i.e.~$\alpha \prec'' \beta$.
The condition of `no-arbitrage' is the absence of arbitrage opportunities in the market $(\mathcal{M}, \sim', \prec'')$.
\end{definition}

Without being able to here go into any details, 
in financial economics or mathematical finance, the so-called Fundamental Theorem of Asset Pricing 
(FTAP; e.g.~\cite{RosSt}, \cite{HaJ}, and \cite{DeFr} in chronological order)
establishes the equivalence of the absence of arbitrage and the existence of a risk-neutral pricing measure (RNPM; for stochastic one-period models)
or an equivalent martingale measure (EMM; for multi-period or time-continuous models).
By means of the risk-neutral pricing formula (a discounted expectation), the RNPM or the EMM implies a linear price functional
for all cash flows which is in line, or consistent, with the existing prices and with the existing objective strict preference order
(more preferable means more expensive). It therefore completes the market. 
Completeness here means that all cash flows are replicable by trading strategies and, thus, have a price. 
Having a price, however, means that any two cash flows can be compared with one another.

I think it is remarkable that, while later articles on FTAP are mostly concerned with the stochastic aspects,
earlier publications (cf.~\cite{RosSt}, \cite{HaJ}) also show some focus on FTAP as an extension theorem of the
pre-existing linear price functional.
As such, also in its modern versions, FTAP is an extension or completion theorem and the {\em a priori} given 
linear trade relation and the strict preference relation
can be completed in a consistent -- which here means arbitrage-free -- manner if and only if the 
original market is arbitrage-free. Moreover, FTAP also states that the RNPM or the EMM, and therefore the new
price functional, is uniquely determined if and only if the {\em a priori} given market is complete.

\begin{theorem}
\label{theo:4oib}
For a financial market $(\mathcal{M}, \sim', \prec'')$, the following are equivalent.
\begin{enumerate}
\item
It is free of arbitrage opportunities.
\item
Its fair exchanges, $\sim'$, and its down-trades, $\prec''$, are chain-consistent with one another. 
\item
Its attainable trades consistently extend / are chain-consistent with its fair exchanges, $\sim'$, and its down-trades, $\prec''$. 
\item
There exists a complete (math.: total) preference order on $\mathcal{M}$ which 
consistently extends the objective indifference relation, $\sim'$, and the objective strict preference order, $\prec''$. 
\end{enumerate}
Under no-arbitrage, the attainable trades are complete (math.: total) if and only if the consistent completion is unique.
\end{theorem}

\begin{proof}
Def.~\ref{def:Consistency} implies that the chain-consistency of $\preccurlyeq_1 \, = \, \sim'$ and $\preccurlyeq_2 \, = \, \prec''$
is equivalent to no-arbitrage (note that here $\prec_1 \, = \, \sim_2 \, = \emptyset$). This is because on the one hand,
an arbitrage trade chain establishes a closed chain as in Def.~\ref{def:Consistency}, where the
at least one $(\gamma_i,\gamma_{i+1})\in\;\prec_1 \cup \prec_2$ is warranted by $\alpha \prec'' \beta$ .
Vice versa, if a chain as in Def.~\ref{def:Consistency} exists, then the at least one 
$(\gamma_i,\gamma_{i+1})\in\;\prec_1 \cup \prec_2$ in fact implies at least one $(\gamma_i,\gamma_{i+1})\in\;\prec''$,
so I can set $\alpha = \gamma_i$ and $\beta = \gamma_{i+1}$.
Theorem \ref{theo:4oib} now is a direct consequence of Corollary \ref{cor:6zrn}.
\end{proof} 

For a distinction of notions, I mention that in \cite{MaAn}, a `complete' preference order is called `rational'. 
In mathematics, however, the descriptors `total' (Def.~\ref{def:relations}) or `strongly connected' are used if all elements of $\mathcal{M}$ can
be compared using the considered preorder. The notion `rational' I would rather use interchangeably to `arbitrage-free'.

Like FTAP, Theorem \ref{theo:4oib} is an extension theorem. It equates no-arbitrage,
which is the consistency of two pre-existing preference orders on a part of the cash flows,
to the existence of a preference order (FTAP: price functional/RNPM/EMM) for all cash flows which is
consistent with the two pre-existing preference orders. 
Moreover, if this extension of the preference orders is uniquely determined, 
the market has a completeness property in the sense that all cash flows were comparable {\em a priori}
and could be traded for one another in at least one direction.
Hence, the parallels of FTAP to Theorem \ref{theo:4oib} are striking, and while my result
stays clear of any linearity assumptions or stochastic models and therefore does not directly imply FTAP in its modern form, 
I consider Theorem \ref{theo:4oib} a broad generalization of it.


\section{Application in microeconomic theory}

In this section, I assume that the preferences of two entities, e.g.~of two individuals, named 1 and 2, are given by the
additive positively homogeneous preorders, i.e.~vector preorders,
$\preccurlyeq_1$ and $\preccurlyeq_2$ on the real linear space $\mathcal{M}$ with the nullvector $o$. 
For instance, one could think of preferences relating to 
appropriately defined deterministic or stochastic cash flows.
Additivity, positive homogeneity, and reflexivity at first, and superficially, appear to be reasonable properties 
of preferences in a financial context. 

If individual 1 holds good $\alpha \in \mathcal{M}$ and 2 holds $\beta \in \mathcal{M}$, then 
$\alpha + \beta \in \mathcal{M}$ is the total amount in goods (e.g.~the total cash flow) that is allocated to 1 and 2.
If for $\alpha', \beta' \in \mathcal{M}$ with $\alpha' + \beta' = \alpha + \beta$ it holds that 
$\alpha \preccurlyeq_1 \alpha'$ and $\beta \preccurlyeq_2 \beta'$, and
at least one $\preccurlyeq_i$ for $i=1,2$ can be replaced with its strict version, $\prec_i$,
then $(\alpha', \beta')$ is called a Pareto improvement of the allocation $(\alpha, \beta)$.
As is well-known, a Pareto improvement therefore improves the situation of at least one individual while not worsening 
the situation of the other.

\begin{theorem}
\label{theo:8pov}
Let $\preccurlyeq_1$ and $\preccurlyeq_2$ be two additive positively homogeneous preorders (vector preorders).
There exists a Pareto improvement for any allocation between two entities with these preferences 
if and only if $\preccurlyeq_1$ and $\preccurlyeq_2$ are not chain-consistent with one another.
\end{theorem}

By Theorem \ref{theo:2qwc}, under the conditions of Theorem \ref{theo:8pov},
chain-consistency is equivalent to the existence of a total additive positively homogeneous preorder
which consistently extends both, $\preccurlyeq_1$ and $\preccurlyeq_2$.
So, expressed differently, {\em if two individuals with additive positively homogeneous preferences
are not both `cut from the same cloth', there exists no Pareto optimal allocation,
and, in this sense, no equilibrium.} Obviously, an analog statement follows for an individual with preferences $\preccurlyeq_1$
vis-\`a-vis an entire market, if the attainable trades of this market as in Def.~\ref{def:24wec} of
Sec.~\ref{sec:Application in financial economics} are given by $\preccurlyeq_2$.
In a loose sense, my result adds to the row of impossibility problems concerning preference orders in microeconomics,
such as Arrow's Impossibility Theorem or the Condorcet Paradox (cf.~\cite{MaAn}).

\begin{proof}[Proof of Theo.~\ref{theo:8pov}]
Generally, for the following, note that for any additive preorder $\preccurlyeq$ on a linear space $\mathcal{M}$,
$\alpha \preccurlyeq \beta$ 
implies 
$\alpha - \beta \preccurlyeq o$
and 
$- \beta \preccurlyeq - \alpha$, since, by additivity and reflexivity, I can consecutively add 
$- \beta \preccurlyeq - \beta$ and then
$- \alpha \preccurlyeq - \alpha$ to $\alpha \preccurlyeq \beta$.
Now, if $(\alpha', \beta')$ is a Pareto improvement of some $(\alpha, \beta) \in \mathcal{M}\times\mathcal{M}$,
then $\alpha' = \alpha + \delta$ and $\beta' = \beta - \delta$ for some $\delta\in\mathcal{M}$,
since the total allocation does not change. With the rules explained earlier and
since $\alpha \preccurlyeq_1 \alpha + \delta$ and $\beta \preccurlyeq_2 \beta - \delta$, a Pareto improvement 
therefore is equivalent to the existence of some $\delta \in \mathcal{M}$ with
\begin{equation}
\label{eq:30rtm}
o \;
\preccurlyeq_1
\; \delta \;
\preccurlyeq_2 \;
o  ,
\end{equation}
such that at least one $\preccurlyeq_i$ for $i=1,2$ in \eqref{eq:30rtm} can be replaced with its strict version, $\prec_i$.\\
$\Rightarrow$:
By Def.~\ref{def:Consistency}, if the preferences are consistent, then a closed chain of length two as just explained
in \eqref{eq:30rtm} is impossible.
Thus, a Pareto improvement implies inconsistency.\\
$\Leftarrow$:
Inconsistency implies for some $n\in\mathbb{N}\setminus\{0,1\}$
that $\gamma_i \preccurlyeq_1 \gamma_{i+1}$ for all $i\in I$ and 
that $\gamma_i \preccurlyeq_2 \gamma_{i+1}$ for all $i\in J$, where $I$ and $J$ are disjoint and nonempty
(cf.~Lemma \ref{lem:7pom}),
$I \cup J = \{1,\ldots,n-1\}$, and $\gamma_1 = \gamma_n$. 
Additivity and reflexivity of $\preccurlyeq_1$  and $\preccurlyeq_2$ now imply 
\begin{equation}
\label{eq:31qtm}
o 
\; \preccurlyeq_1 \;
\sum_{i\in I} (\gamma_{i+1} - \gamma_i) ,
\end{equation}
as well as
\begin{equation}
\label{eq:32qtm}
\sum_{i\in J} (\gamma_i - \gamma_{i+1}) 
\; \preccurlyeq_2 \;
o .
\end{equation}
Because of the requirements of Def.~\ref{def:Consistency} and by the second statement of Lemma \ref{lem:1csv},
\eqref{eq:31qtm}, \eqref{eq:32qtm}, or both, must be strict.
However, 
\begin{eqnarray}
\sum_{i\in J} (\gamma_i - \gamma_{i+1})
& = &
\sum_{i=1}^{n-1} (\gamma_i - \gamma_{i+1})
-
\sum_{i\in I} (\gamma_i - \gamma_{i+1})  \\
\nonumber
& = &
\gamma_1 - \gamma_n
+
\sum_{i\in I} (\gamma_{i+1} - \gamma_i) . \\
\nonumber
& = &
\sum_{i\in I} (\gamma_{i+1} - \gamma_i) ,
\end{eqnarray}
and, with $\delta := \sum_{i\in I} (\gamma_{i+1} - \gamma_i)$, one obtains a chain as in \eqref{eq:30rtm}
from \eqref{eq:31qtm} and \eqref{eq:32qtm}.
\end{proof}

In financial economics, or mathematical finance, 
creating a proper financial advantage out of nothing is called a `free lunch' or an `arbitrage'
(see also Sec.~\ref{sec:Application in financial economics}). Modern financial
mathematics is based on the principle of `no-arbitrage', the exclusion of arbitrage for the reason that such opportunities 
quickly disappear in real markets.
In finance, it matters what exactly `proper' means, because there it needs to be an objective notion, while in the situation 
at hand here, it does not necessarily, since I can consider the preferences of two individuals, and all that matters is what 
the individuals prefer over the zero, $o$, and not, for instance, what a market summarily prefers.

Reconsidering the property of additivity, in a financial context, this may not be the wisest of properties for individual preferences. 
In risk management or insurance, diversification, which can stem from stochastic independence, has always been a very important aspect.
If $\alpha \preccurlyeq_i \beta$ and $\alpha' \preccurlyeq_i \beta'$, but $\alpha$ and $\alpha'$ enabled diversification by independence,
while $\beta$ and $\beta'$ do not (and maybe do quite the opposite when lumped together), preferring $\beta + \beta'$
over $\alpha + \alpha'$ could be a bad idea. Theorem \ref{theo:8pov} provides a further reason why, generally, additive preferences 
are problematic and, so I presume, unrealistic.





\begin{thebibliography}{13}
\bibitem{CaRe}
Castillo, R.~E. (2005): A note on Krein's theorem. {\it Lecturas Matem\'aticas}, \textbf{26} (1), 5--9.
\bibitem{DeFr}
Delbaen, F., and W.~Schachermayer (1994): A general version of the fundamental theorem of asset pricing. {\it Mathematische Annalen}, \textbf{300} (3), 463--520.
\bibitem{HaJ}
Harrison, J.~M., and D.~M.~Kreps (1979): Martingales and Arbitrage in Multiperiod Securities Markets. {\it Journal of Economic Theory}, \textbf{20} (3), 381--408.
\bibitem{MaHo}
MacNeille, H.~M. (1937): Partially Ordered Sets. {\it Transactions of the American Mathematical Society}, \textbf{42} (3), 416--460.
\bibitem{MaAn}
Mas-Colell, A., M.~Whinston, and J.~Green (1995): {\it Microeconomic Theory}. Oxford University Press.
\bibitem{NaLa}
Narici, L., and E.~Beckenstein (2010): {\it Topological Vector Spaces}. Second Edition. CRC Press.
\bibitem{RoSt}
Roman, S. (2008): 
{\it Lattices and Ordered Sets}. Springer.
\bibitem{RosSt}
Ross, S.~A. (1977): Return, Risk, and Arbitrage. In: {\it Risk and Return in Finance} (edited by I.~Friend and J.~L.~Bicksler). Ballinger Publishing Company.
\bibitem{ScHe}
Schaefer, H.~H., and M.~P.~Wolff (1999): {\it Topological Vector Spaces}. Second Edition. Springer.
\bibitem{ScBe}
Schr\"oder, B. (2016): 
{\it Ordered Sets}. Second Edition. Birkh\"auser.
\bibitem{ZoMa}
Zorn, M. (1935): A remark on method in transfinite algebra. {\it Bulletin of the American Mathematical Society}, \textbf{41} (10), 667--670.
\end{thebibliography}
\end{document}